\documentclass{amsart}
\usepackage{hyperref}
\usepackage{amsmath}
\usepackage{amsthm}
\usepackage{cleveref}
\usepackage{amsfonts}
\usepackage{graphicx}
\usepackage{epstopdf}
\usepackage{algorithmic}
\usepackage{xargs}
\usepackage{soul}
\usepackage{color}
\usepackage{amssymb} 

\usepackage{soulutf8}
\usepackage{diagbox}
\usepackage{subcaption}
\usepackage{xargs}
\usepackage{rotating}
 \usepackage{amsaddr}
\ifpdf
  \DeclareGraphicsExtensions{.eps,.pdf,.png,.jpg}
\else
  \DeclareGraphicsExtensions{.eps}
\fi
\usepackage[disable, colorinlistoftodos,prependcaption, textsize=tiny]{todonotes}

\newcommand{\cen}{\mathrm{cen}}
\newcommand{\col}{\mathrm{col}}
\newcommand{\e}{\mathrm{e}}
\newcommand{\He}{ H^{\gamma, \varepsilon}}
\def\eps{\varepsilon}
\newcommand{\x}{\mathbf{x}}
\newcommand{\tb}{\mathbf{t}}
\newcommand{\kb}{\mathbf{k}}
\newcommand{\lb}{\mathbf{\ell}}
\newcommand*\lbb{\ensuremath{\boldsymbol\ell}}
\newcommand{\ve}{\text{vec}}
\newtheorem{remark}{Remark}
\renewcommand{\epsilon}{\varepsilon}

\numberwithin{remark}{section}
\title[Stable evaluation of Gaussian RBF using Hermite polynomials]{Stable evaluation of Gaussian radial basis functions using Hermite polynomials}
\author{by Anna Yurova$^{1,2}$ and Katharina Kormann$^{1,2}$}
\address{$^1$Max-Planck-Institut f\"ur Plasmaphysik, 85748 Garching, Germany(\url{anna.yurova@ipp.mpg.de}, \url{katharina.kormann@ipp.mpg.de}).}
\address{$^2$Technische Universitat M\"unchen, Zentrum Mathematik, 85748 Garching, Germany} 
\date{\today}
\thanks{\textbf{Funding:} This work has been carried out within the framework of the EUROfusion Consortium and has received funding from the Euratom research and training programme 2014-2018 under grant agreement No 633053. The views and opinions expressed herein do not necessarily reflect those of the European Commission.}

\def\eps{\varepsilon}

\newtheorem{theorem}{Theorem}
\newtheorem{lemma}{Lemma}
\numberwithin{theorem}{section}
\numberwithin{lemma}{section}
\usepackage{color}
\begin{document}
\maketitle

\begin{abstract}
 Gaussian radial basis functions can be an accurate basis for multivariate interpolation. In practise, high accuracies are often achieved in the flat limit where the interpolation matrix becomes increasingly ill-conditioned. Stable evaluation algorithms have been proposed by Fornberg, Larsson \& Flyer based on a Chebyshev expansion of the Gaussian basis and by Fasshauer \& McCourt based on a Mercer expansion with Hermite polynomials. In this paper, we propose another stabilization algorithm based on Hermite polynomials but derived from the generating function of Hermite polynomials. The new expansion does not require a complicated choice of parameters and offers a simple extension to high-dimensional tensor grids as well as a generalization for anisotropic multivariate basis functions using the Hagedorn generating function.
\end{abstract}

\section{Introduction}

Multivariate interpolation is a topic of recent interest, for instance appearing in the semi-Lagrangian solution of high-dimensional advection problems. Gaussian radial basis function interpolation generalizes to higher dimensions in a simple way and can yield spectral accuracy \cite{fasshauer2012}. However, it is known that rather small values of the \emph{shape parameter} (width of the Gaussian) are often required for optimal accuracy. In this case the basis functions become increasingly flat and the interpolation matrix becomes ill-conditioned. Tarwater has described this phenomenon in 1985~\cite{tarwater1985} and the problem has been extensively studied in the literature (see \cite{fornberg_flyer_2015} for a review). The eigenvalues of the interpolation matrix are proportional to increasing powers of the shape parameter as has been quantified by Fornberg and Zuev \cite{fornberg2007}.

A direct collocation solution of the interpolation problem, referred to as RBF-Direct in the literature, computes the expansion coefficients of the Gaussian interpolant by inverting the collocation matrix and then evaluating the expansion. This procedure suffers from inaccuracies in floating point arithmetics due to the ill-conditioning of the matrices. In recent years, several algorithms have been proposed to stabilize the computations of the radial basis functions interpolation problems.  These stabilization algorithms directly evaluate the interpolant in a sequence of well-conditioned steps by a transformation to a different basis. 
The first method was the Contour--Pad\'e approximation proposed by Fornberg and Wright  for multiquadrics \cite{fornberg2004}. Later Fornberg and Piret \cite{fornberg2007stable} proposed the so-called RBF-QR method for stable interpolation with Gaussians on the sphere. The Gaussian basis is expanded in spherical harmonics. The expansion allows to isolate the ill-conditioning in a diagonal matrix that can be inverted in a well-conditioned procedure.

The method has been extended to more general domains in one to three dimensions by Fornberg, Larsson \& Flyer \cite{fornberg2011stable}. This expansion is based on a combination of Chebyshev polynomials and spherical harmonics. This method will be referred to as Chebyshev-QR in this paper. The technique has also been used for the stable computation of difference matrices by Larsson et al.~\cite{larsson2013} and by Fornberg et al.~\cite{fornberg2016} for RBF-FD stencils. In order to treat complex domains, the Chebyshev-QR method has been combined with a partition of unity approach by Larsson, Shcherbakov \& Heryudono \cite{larsson2017}. 

Fasshauer and McCourt \cite{fasshauer2012stable} have developed another RBF-QR method, called Gauss-QR, that relies on a Mercer expansion of the Gaussian kernel. The basis transformation involves exponentially scaled Hermite polynomials. Compared to the Chebyshev-QR method by Fornberg, Larsson \& Flyer \cite{fornberg2011stable}, the Gauss-QR algorithm extends to higher dimensions in a simpler way and does not require transformation of the computational domain into the unit squere. On the other hand, the method introduces an additional parameter that needs to be hand-tuned. In this paper, we propose an expansion built on Hermite generating functions. Our new basis is similar to the one in  \cite{fasshauer2012stable} with the difference that the introduced parameter can easily be chosen. Our focus is on enabling high-dimensional interpolation where we propose a tensor product approach that yields a memory-sparse representation of the interpolation matrices. 
Moreover, we propose a stabilization algorithm for anisotropic multivariate Gaussians: We adopt the framework of Hagedorn wave packets from the semi-classical quantum dynamics literature \cite{lubich2008quantum, hagedorn1998raising}. Hagedorn wave packets are combinations of multivariate versions of Hermite polynomials and anisotropic Gaussians. Analogously to the Hermite polynomials, Hagedorn generating functions can be considered (see \cite{dietert2017invariant,hagedorn2015generating}) that enable a generalization of our Hermite expansion to the anisotropic case.

The paper is organized as follows: In the next section we introduce our HermiteGF expansion of the radial basis functions and discuss its convergence. In \cref{sec:stabilization}, we discuss two main ideas of truncating the expansion: one based on a direct transform to the HermiteGF basis and another following the RBF-QR idea. Extentions to multivariate interpolation are discussed in \cref{sec:multivariate}. Numerical results show the accuracy of our method in \cref{sec:numericalResults} and computational complexity and performance are discussed in \cref{sec:performance}. Finally, \cref{sec:conclusions} concludes the paper.

\section{HermiteGF expansion}
\label{sec:directInterpHermite}

In this section, similarly to \cite{fornberg2007stable,fornberg2011stable,fasshauer2012stable}, we propose an expansion of the radial basis functions in a ``better'' basis, that spans the same space, but avoids instabilities related to the flat limit.

\subsection{Interpolation problem}
Before introducing our expansion of the Gaussian basis, let us briefly define the interpolation problem in one dimension.
\label{sec:interpolationProblem}
Given a set $\{\phi_k(x)\}_{k=1}^N$ of basis functions and the values $\{f_i\}$ of the function $f$ at points $\{x_i^{\mathrm{col}}\}_{i=1}^N$ we seek to find an interpolant of the following form,
\begin{equation}
 s(x) = \sum_{k = 1}^N \alpha_k \phi_k(x), 
 \label{eq:RBFinterpolation}
\end{equation}
such that it satisfies the $N$ collocation conditions,
\begin{equation}
 s(x_i^{\mathrm{col}}) = f_i \quad \text{for} \quad i=1,\ldots, N.
\end{equation}
The straightforward approach is to find the coefficients $\{\alpha_i\}$ as a solution of the linear system,
\begin{equation}
 \Phi^{\mathrm{col}} \alpha = f, \quad \text{with}\quad \Phi_{ij}^{\mathrm{col}} = \phi_j(x_i^{\mathrm{col}}).
 \label{eq:interpolationProblem}
\end{equation}
The matrix $\Phi^{\mathrm{col}}$ is called \emph{collocation matrix}. Then, the interpolant \cref{eq:RBFinterpolation} can be evaluated at any point of the domain.

Here we focus on Gaussian radial basis functions,
\begin{equation}
\phi_k(x) = \exp(-\epsilon^2 \Vert x - x_k^{\cen} \Vert ^2), 
\end{equation}
with shape parameter  $\epsilon > 0$. 

\subsection{Definition}
Let $\{h_n\}_{n \geq 0}$ be the Hermite polynomials in the physicists' version, that satisfy the following recurrence relation,
\begin{equation}
 h_{n+1}(x) = 2xh_n(x) -2nh_{n-1}(x).
\end{equation}
The following upper bound holds for the magnitude of Hermite polynomials \cite[Expression 22.14.17]{abramowitz1964handbook},
\begin{equation}
 \vert h_n(x) \vert \leq e^{\frac{x^2}{2}}c2^{\frac{n}{2}}\sqrt{n!}, \quad c \approx 1.086435.
 \label{eq:HermiteUpperBound}
\end{equation}
The factors $\sqrt{n!}$, $2^{n/2}$ grow very fast with $n$. Therefore, in order to avoid overflow for large $n$, it is advantageous for numerical computations to scale the Hermite polynomials with the factor $\sqrt{2^n n!}$.
Hence, let us define the following basis functions,
\begin{equation}
\He_n(x) = \frac{1}{\sqrt{2^n n!}}h_n(\gamma x)\e^{-\varepsilon^2 x^2}, \quad \varepsilon>0, \, \gamma>0,
\end{equation}
that we refer to as \emph{HermiteGF functions}.
Based on the generating function theory we derive an infinite expansion of the one dimensional Gaussian RBFs in the new HermiteGF basis $\{\He_n\}$.

\begin{theorem}{HermiteGF expansion}

 \noindent
For all $\varepsilon > 0$, $\gamma > 0$, $y \in \mathbb{R}$, we have a pointwise expansion
\begin{equation}
\phi_y(x) = \e^{-\varepsilon^2\left(x-y\right)^2} = \exp\left(\varepsilon^2 y^2\left(\frac{\varepsilon^2}{\gamma^2} - 1\right)\right) \sum_{n \geq 0} \frac{\varepsilon^{2n}\sqrt{2^n}}{\gamma^n\sqrt{n!}}y^n \He_n(x).
\label{eq:RBFhermiteExpansionScaled}
\end{equation}
The RBF interpolant $s(x)$ can then be pointwise computed as,
\begin{equation}
 s(x) = \sum_{k = 1}^N \alpha_k\exp\left(\varepsilon^2 (x^{\cen}_k)^2\left(\frac{\varepsilon^2}{\gamma^2} - 1\right)\right)  \sum_{n \geq 0}  \frac{\varepsilon^{2n} \sqrt{2^n}}{\gamma^n \sqrt{n!}}(x^{\cen}_k)^n \He_n(x),
 \label{eq:interpolantExpansion}
\end{equation}
where $\{x_k^{\cen}\}_{k=1}^N$ are the centers of the RBFs.
\label{th:HermiteGF_expansion}
\end{theorem}
\begin{proof}
The Hermite polynomial's generating function is given by (see e.g.~\cite[Expression 22.9.17]{abramowitz1964handbook}),
 \begin{equation}
  \mathrm{e}^{2st - t^2} = \sum_{n \geq 0} \frac{t^n}{n!} h_n(s) 
 \end{equation}
Choosing $t = \frac{\varepsilon^2 y}{\gamma}$ and $s = \gamma x$, we obtain
\begin{equation}
 \sum_{n \geq 0} \frac{\varepsilon^{2n}}{\gamma^n n!}y^nh_n(\gamma x) = \exp\left(2\varepsilon^2y x - \frac{\varepsilon^4 y^2}{\gamma^2}\right).
\end{equation}
Hence, we get
\begin{align}
&\phantom{=} \exp\left(\varepsilon^2 y^2\left(\frac{\varepsilon^2}{\gamma^2}-1\right)\right) \sum_{n \geq 0}\frac{\varepsilon^{2n} \sqrt{2^n}}{\gamma^n \sqrt{n!}}y^n\He_n(x) \\
 &= \exp\left(\varepsilon^2 y^2\left(\frac{\varepsilon^2}{\gamma^2}-1\right)\right) \sum_{n \geq 0}\frac{\varepsilon^{2n} }{\gamma^n n!}y^nh_n(\gamma x)\mathrm{e}^{-\varepsilon^2x^2} \\
 &= \exp\left(\varepsilon^2y^2\left(\frac{\varepsilon^2}{\gamma^2}-1\right) + 2\varepsilon^2y x - \frac{\varepsilon^4y^2}{\gamma^2} - \varepsilon^2x^2\right) = \mathrm{e}^{-\varepsilon^2(x - y)^2},
\end{align}
which proves expansion \eqref{eq:RBFhermiteExpansionScaled}. Using expansion \eqref{eq:RBFhermiteExpansionScaled} in the interpolant \eqref{eq:RBFinterpolation}, we get the representation \eqref{eq:interpolantExpansion}.
\end{proof} 

\subsection{Basis centering}
The Hermite polynomials are symmetric with respect to the axis $x=0$. Due to the growth in the basis it is advantageous to center the interpolation interval $[A, B]$ at 0. 
For this reason, we symmetrize the  basis around $x_0 := \frac{A+B}{2}$. The RBF $\phi_k(x)$ can be expanded as,
\begin{align}
 \phi_k(x) &= \mathrm{e}^{-\varepsilon^2(x-x_k^{\cen})^2} = \mathrm{e}^{-\varepsilon^2(x - x_0 - (x_k^{\cen} - x_0))^2} \\
 &=  \mathrm{e}^{\left(\varepsilon^2 (x_k^{\cen}-x_0)^2\left(\frac{\varepsilon^2}{\gamma^2} - 1\right)\right)} \sum_{n \geq 0} \frac{\varepsilon^{2n}\sqrt{2^n}}{\gamma^n\sqrt{n!}}(x_k^{\cen}-x_0)^n \He_n(x-x_0).
\end{align}

Then, we have,
\begin{equation}
 x-x_0 \in \left[ - \frac{B-A}{2}, \frac{B-A}{2} \right],
\end{equation}
i.e.~the HermiteGF functions $\He_n$ are evaluated on an interval centered around 0. 
For the sake of simplicity, we further consider symmetric intervals $[-L, L]$. However, the procedure can be applied to functions on arbitrary intervals by adding this translation by $x_0$.
 \subsection{The parameter $\gamma$}
 The parameter $\gamma$ in the basis $\{\He_n\}$ allows a control over the evaluation domain of the Hermite polynomials. When choosing $\gamma$, one has to consider two counteracting effects: For small values of $\gamma$, the collocation points are close which can yield ill-conditioning since the values of the basis functions at the collocation points are too similar. On the other hand, Hermite polynomials take very large values on large domains which can lead to an overflow. An optimal balance depends on the particular function and the number of basis functions. However, from our numerical experience, choosing $\gamma L$ between 3 and 5 yields good approximation quality in most cases (cf.~\cref{sec:scalAndCond}). 
 
\subsection{Connection to Fasshauer and McCourt}

An expansion of similar type was used by Fasshauer and McCourt \cite{fasshauer2012stable} for the stabilization of the RBF interpolation. Instead of the HermiteGF-expansion, an eigenfunction expansion of Gaussian RBF was used. The corresponding eigenfunctions look as follows \cite[$\mathsection$ 3.1]{fasshauer2012stable},
\begin{equation}
 \phi_n(x) = \frac{\sqrt{\beta}}{\sqrt{2^nn!}} \exp(-\delta^2 x^2) h_{n-1}(\alpha \beta x),
\end{equation}
The parameter $\alpha$ needs to be chosen by the user. Then, the parameters $\beta$ and $\delta$ are deduced from $\alpha$ and $\eps$ according to the formula,
\begin{equation}
\beta = \left( 1 + \frac{4\varepsilon^2}{\alpha^2}\right)^{1/4}, \quad \delta^2 = \frac{\alpha^2}{2}(\beta^2 - 1).
\label{eq:fasshauerParameterRelations}
\end{equation}
We now try to match that basis with the basis functions arising from the HermiteGF expansion. To match the width of the exponential in the two expansions we need,
\begin{equation}\label{eq:fasshauer_match_delta}
 \delta = \varepsilon
\end{equation}
and to match the argument of the Hermite polynomials it is necessary to have,
\begin{equation}
 \alpha \beta = \gamma.
 \label{eq:alphaBetaGamma}
\end{equation}
We now compute the values of the parameters $\alpha, \beta$ from the relations \eqref{eq:fasshauerParameterRelations},
\begin{equation}
 \varepsilon^2 = \frac{\gamma^2 - \alpha^2}{2} \implies \alpha = \sqrt{\gamma^2 - 2\varepsilon^2}.
\end{equation}
The parameter $\beta$ can then be calculated as 
\begin{equation}
\beta = \left( 1 + \frac{4\varepsilon^2}{\alpha^2}\right)^{1/4} = \left( 1 + \frac{4\varepsilon^2}{\gamma^2 - 2\varepsilon^2}\right)^{1/4}
\label{eq:beta1}
\end{equation}
However, from the relation \eqref{eq:alphaBetaGamma} $\beta$ must be,
\begin{equation}
 \beta = \frac{\gamma}{\alpha} = \frac{\gamma}{\sqrt{\gamma^2 -2\varepsilon^2}}
 \label{eq:beta2}
\end{equation}
One can see that if $\varepsilon \rightarrow 0$, both expressions for $\beta$ converge to $1$. However in a general case the values of expressions \eqref{eq:beta1} and \eqref{eq:beta2} for the parameter $\beta$ differ. Hence, we cannot match both \eqref{eq:fasshauer_match_delta} and \eqref{eq:alphaBetaGamma} at the same time. This means that there is no direct correspondence between the basis functions arising from the HermiteGF expansion and the ones used by Fasshauer and McCourt \cite{fasshauer2012stable}. 

\subsection{Convergence of the truncated HermiteGF expansion}
In this section, we check the convergence of the expansion \eqref{eq:interpolantExpansion}, 
if we cut the expansion \eqref{eq:interpolantExpansion} after $M$ terms,
\begin{equation}
 s(x) \approx s_M^{\gamma}(x) := \sum_{k=1}^N \alpha_k \exp\left(\varepsilon^2 (x^{\cen}_k)^2\left(\frac{\varepsilon^2}{\gamma^2} - 1\right)\right) \sum_{n=0}^{M-1} \frac{\varepsilon^{2n} \sqrt{2^n}}{\gamma^n \sqrt{n!}}(x_k^{\cen})^n \He_n(x).
 \label{eq:truncHermiteInterp}
\end{equation}
We later refer to $s_M^{\gamma}(x)$ as \emph{HermiteGF interpolant}. Let us now prove that for a large enough $M$ the approximation $s_M^{\gamma}(x)$ converges to $s(x)$.
\begin{theorem}

Let $s$ be the RBF interpolant,
\begin{equation}
 s(x) = \sum_{k=1}^N \alpha_k \phi_k(x) = \sum_{k=1}^N \alpha_k \e^{-\eps^2(x - x_k^{\cen})^2}
\end{equation}
with $\{x_k^{\cen}\}_{k=1}^N \subset [-L, L]$.

For all $x \in [-L, L]$, the HermiteGF interpolant $s_M^{\gamma}(x)$ given by \eqref{eq:truncHermiteInterp}  converges pointwise to $s(x)$, i.e.
\begin{equation}
 \vert s(x) - s_M^{\gamma}(x) \vert \rightarrow 0 \quad \text{for} \quad M \rightarrow \infty
\end{equation}
For $\gamma > \sqrt{2}\eps^2 L$, we also have the estimate
\begin{equation}
 \vert s(x) - s_M^{\gamma}(x) \vert < C \frac{q^M}{(1-q)\sqrt{M!}},
 \label{eq:geomEstimation}
\end{equation}
where $q = \frac{\sqrt{2}\eps^2 L}{\gamma}$ and $C = C(\gamma, \eps, L, \{\alpha_k\}) \in \mathbb{R}$ is a constant.
\end{theorem}
\begin{proof}
We construct the proof analogously to \cite[$\mathsection$ 3.1]{rashidinia2016stable}. Combining \eqref{eq:interpolantExpansion} and \eqref{eq:truncHermiteInterp}, for each $x$ we have,
\begin{equation}
 \vert s(x) - s_M(x) \vert = \left\vert\sum_{k=1}^N\alpha_k \sum_{n=M}^{\infty}\exp\left(\varepsilon^2 (x^{\cen}_k)^2\left(\frac{\varepsilon^2}{\gamma^2} - 1\right)\right)\frac{\varepsilon^{2n}}{\gamma^n n!} (x^{\cen}_k)^n h_n(\gamma x)\mathrm{e}^{-\varepsilon^2 x^2}\right\vert.
\end{equation}
Denoting $\mathcal{A} = \mathrm{max}\{\vert \alpha_j \vert, j=1,..., N\}$ and using the upper bound for the $n$-th Hermite polynomial \eqref{eq:HermiteUpperBound} we obtain,
\begin{multline}
 \vert s(x) - s_M^{\gamma}(x) \vert \leq \mathcal{A} \sum_{k=1}^N \sum_{n=M}^{\infty}{\Bigg (}\exp\left(\varepsilon^2 (x^{\cen}_k)^2\left(\frac{\varepsilon^2}{\gamma^2} - 1\right)\right) \frac{\varepsilon^{2n}}{\gamma^n n!} \vert x^{\cen}_k \vert ^n \cdot \\ \cdot \mathrm{e}^{\left(\frac{\gamma^2 }{2} - \eps^2\right)x^2}c 2^{\frac{n}{2}}\sqrt{n!}{\Bigg )}.
\end{multline}
To further estimate this expression, we use that $\vert x_k^{\cen} \vert \leq L$, $k=1,\ldots,N$, and $|x|\leq L$ and introduce the constants,
\begin{equation}
      P_1 =  \max \left\{\exp\left(\left(\frac{\gamma^2}{2} - \varepsilon^2\right) L^2\right), 1\right\}, \quad P_2 = \max \left\{ \exp\left(\varepsilon^2 L^2\left(\frac{\varepsilon^2}{\gamma^2} - 1\right)\right),1 \right\}.
\end{equation}    
Then, we obtain the bound
\begin{align}
\vert s(x) - s_M^{\gamma}(x) \vert &\leq \underbrace{ \vphantom{\sum_{n=1}^{\infty}} c \mathcal{A} N P_1P_2}_{C}  \underbrace{\sum_{n=M}^{\infty} \frac{(\sqrt{2}\varepsilon^2 L)^{n}}{\gamma^n \sqrt{n!}}}_{T_{M}}
\end{align}
Consider the following series of positive terms,
\begin{equation}
  \sum_{n=0}^{\infty} \underbrace{\frac{(\sqrt{2}\varepsilon^2 L)^{n}}{\gamma^n \sqrt{n!}}}_{t_{n}}.
  \label{eq:seriesHermiteEstimation}
\end{equation}
 Then $T_{M}$ is the tail of the series. Therefore it is enough to prove that the series \eqref{eq:seriesHermiteEstimation} converges in order to prove that $T_{M} \rightarrow 0$ \cite[6.11]{binmore1982mathematical}. It can be shown that $ \lim_{n \rightarrow \infty} t_{n+1}/t_n = 0$ and hence, the series $\{t_n\}$ converges by the ratio criterion \cite[6.17]{binmore1982mathematical}. Therefore, 
 \begin{equation}
  \vert s(x) - s_M^{\gamma}(x) \vert \leq C T_{M} \rightarrow 0 \quad \text{for} \quad M \rightarrow \infty,
 \end{equation}
 where $C$ depends only on the size of the interpolation interval $L$, the coefficients $\{\alpha_k\}$ of the RBF interpolant, the number of RBFs $N$, and the parameters $\varepsilon, \gamma$.
 
 For $\gamma > \sqrt{2}\eps^2 L$, $T_M$ can be estimated by 
 \[ T_M < \frac{1}{\sqrt{M!}} \sum_{n=M}^{\infty}q^n\] 
 with $q = \frac{\sqrt 2 \eps^2 L}{\gamma} < 1$.
 Using the geometric series we obtain \eqref{eq:geomEstimation}.
\end{proof}
\begin{remark}
 Analogously, it can be proven that the HermiteGF interpolant $s_M^{\gamma}(x)$ converges to $s(x)$ in $L_2([-L,L])$. Moreover, the geometric bound \eqref{eq:geomEstimation} holds with a different constant for the $L_2([-L,L])$ norm.
\end{remark}
\section{Stabilization of the RBF interpolation}\label{sec:stabilization}
In this section, we derive a numerical stabilization algorithm for RBF interpolation based on the HermiteGF expansion. The main idea is to perform a basis transformation to a more stable basis $\{\He_n\}$. For appropriately chosen parameter $\gamma$ we expect the basis $\{\He_n\}$ to be better conditioned. We can write the expansion \eqref{eq:RBFhermiteExpansionScaled} as an infinite matrix-vector product,
\begin{equation}
 \begin{pmatrix}
  \phi_1(x),&\ldots,&\phi_N(x)
 \end{pmatrix} =
 {\begin{pmatrix}
  \He_0(x), &
  \ldots, &
  \He_M(x), &
  \ldots &
 \end{pmatrix}}
 B(\eps, \gamma, X^{\cen})
 \label{eq:expansionMatrixForm}
\end{equation}
with
\begin{equation}
 B(\eps, \gamma, X^{\cen})_{nk} = \exp\left(\varepsilon^2 (x_k^{\cen})^2\left(\frac{\varepsilon^2}{\gamma^2} - 1\right)\right) \frac{\varepsilon^{2n}\sqrt{2^n}}{\gamma^n \sqrt{n!}}(x_k^{\cen})^n.
\end{equation}

The major part of the ill-conditioning is now confined in the matrix $B$. Since $B$ is independent of the point $x$ where the basis function is evaluated, both the evaluation and interpolation matrix can be expressed in the form \eqref{eq:expansionMatrixForm} with the same matrix $B$. For this reason, a strategy of dealing with the ill-conditioning in $B$ analytically can be developed. 

To make the representation \eqref{eq:expansionMatrixForm} usable for numerical computations, one has to cut the expansion \eqref{eq:RBFhermiteExpansionScaled} after a certain number of terms $M$. This point has to be chosen such that the order of magnitude of the interpolation error is the same order as the error of the RBF interpolant.

We now consider two ways of dealing with the matrix $B$. One way is to eliminate the matrix $B$ from the computation completely by choosing $M=N$ as proposed in \cite[$\mathsection$ 3.1]{rashidinia2016stable}. This case corresponds to an interpolation in the HermiteGF basis. Even though this method provides good results, it lacks the flexibility of choosing $M$. To allow $M>N$, an RBF-QR algorithm can be designed for the HermiteGF expansion analogously to the Chebyshev RBF-QR algorithm by Fornberg et al.~\cite{fornberg2011stable}.
\subsection{HermiteGF interpolant}
Let us write the RBF interpolant $s(x)$ in the matrix-vector form,
\begin{equation}
 s(x) = \sum_{k=1}^N \alpha_k \phi_k(x) = \Phi(x, X^{\mathrm{cen}}) \alpha,
\end{equation}
where $\Phi(x, X^{\mathrm{cen}}) = \begin{pmatrix}
  \phi_1(x),&\ldots,&\phi_N(x)
 \end{pmatrix}$, $X^{\mathrm{cen}}$ are the centering points of the basis functions and $\alpha$ is the coefficients vector. We now use the expansion~\eqref{eq:RBFhermiteExpansionScaled},
\begin{equation}
 s(x) = \Phi(x, X^{\mathrm{cen}}) \alpha \approx \He(x) B(\varepsilon, \gamma,X^{\mathrm{cen}}) \alpha,
 \label{eq:radialInterpolationHermite}
\end{equation}
where $\He =( \He_0(x), \, \ldots, \, \He_{M-1}(x) )$.
The ill-conditioning related to varying powers of $\varepsilon$ is confined in a matrix $B$.

The system \eqref{eq:interpolationProblem} then takes the form,
\begin{equation}
 f(X^{\mathrm{col}}) = \He(X^{\mathrm{col}}) B(\varepsilon,\gamma, X^{\mathrm{cen}}) \alpha,
\end{equation}
where $X^{\mathrm{col}}$ are the collocation points.
Considering $M = N$ we arrive to the following expression for the coefficients $\alpha$,
\begin{equation}
 \alpha = B(\varepsilon,\gamma, X^{\mathrm{cen}})^{-1}\He( X^{\mathrm{col}})^{-1}f(X^{\mathrm{col}}).
 \label{eq:alphaComputationCollocation}
\end{equation}
If we now insert the expression~\eqref{eq:alphaComputationCollocation} into~\eqref{eq:radialInterpolationHermite}, we get,
\begin{align}
 s(x) &\approx s_M^{\gamma} = \He(x)B(\varepsilon,\gamma, X^{\mathrm{cen}}) B(\varepsilon,\gamma, X^{\mathrm{cen}})^{-1}\He(X^{\mathrm{col}})^{-1}f(X^{\mathrm{col}}) \\
 &= \He(x) \He(X^{\mathrm{col}})^{-1}f(X^{\mathrm{col}}).
 \label{eq:directInterp1D}
\end{align}

The only restriction that we put on the collocation points is that their number should be equal to the number of center points. Note that the obtained expression for the interpolant $s$ \emph{does not depend} on the \emph{grid of centers} $X^{\cen}$. This way of computing $s$ is very easy to implement and allows to avoid ill-conditioning arising in $B$. However, it restricts us to $M=N$.
\subsection{RBF-QR}
In case we want to cut the expansion~\eqref{eq:RBFhermiteExpansionScaled} at $M > N$, the interpolation algorithm gets more complicated. Since the matrix $B$ is now rectangular, $B^{-1}$ is not well defined. Therefore, it is necessary to come up with another way of dealing with the ill-conditioning contained in $B$. We follow the RBF-QR approach and further split $B$ into a well-conditioned full matrix $C$ and a diagonal matrix $D$, where all harmful effects are confined in $D$. In the case of expansion~\eqref{eq:RBFhermiteExpansionScaled}, the following setup follows naturally from the Chebyshev-QR theory \cite[$\mathsection$ 4.1.3]{fornberg2011stable},
\begin{align*}
C_{kn} = \exp\left(\epsilon^2 (x_k^{\cen})^2\left(\frac{\epsilon^2}{\gamma^2} - 1\right)\right) (x_k^{\cen})^n, \quad D_{nn} = \frac{\epsilon^{2n}\sqrt{2^n}}{\gamma^n \sqrt{n!}}.
\end{align*}

A problem is arising when we take \emph{center points} with an absolute value greater than 1. That can lead to an ill-conditioning in $C$. One of the ways to treat this effect is to divide each coefficient by the width of the domain $L$ containing the centering points. That might be dangerous when the domain is too large, however, it still extends the range of available domains. The coefficients then look as follows,
\begin{align*}
C_{kn} = \exp\left(\epsilon^2 (x_k^{\cen})^2\left(\frac{\epsilon^2}{\gamma^2} - 1\right)\right) \frac{(x_k^{\cen})^n}{L^n}, \quad D_{nn} = \frac{\epsilon^{2n}\sqrt{2^n}}{\gamma^n \sqrt{n!}} L^n.
\end{align*}

The goal is to find a basis $\{\psi_j\}$ spanning the same space as $\{\phi_k\}$ but yielding a better conditioned collocation matrix. In particular, we need an invertible matrix $X$ such that $X^{-1}\Phi^T$ is better conditioned. 
Let us perform a QR-decomposition on $C = QR$. Then, we get,
\begin{equation}
 \Phi(x)^T = CD\He(x)^T = Q
 \begin{pmatrix}
 R_1 & R_2                     
 \end{pmatrix}
 \begin{pmatrix}
  D_1 & 0 \\
  0 & D_2
 \end{pmatrix}\He(x)^T.
 \label{eq:CDPsi}
\end{equation}
Consider $X = QR_1D_1$. The new basis $\Psi := X^{-1} \Phi(x)^T$ can be formed as,
\begin{align}
\Psi(x)^T & = D_1^{-1} R_1^{-1} Q^{\mathrm{H}}\Phi(x)^T = D_1^{-1} R_1^{-1} Q^{\mathrm{H}} Q
 \begin{pmatrix}
  R_1 D_1 &
  R_2 D_2
 \end{pmatrix}\He(x)^T\\
&= \begin{pmatrix}
   \mathrm{Id} &
   D_1^{-1}R_1^{-1}R_2 D_2
  \end{pmatrix}
\He(x)^T.
\label{eq:correctionRelation}
\end{align}
To avoid under/overflow in the computation of $D_1^{-1}R_1^{-1}R_2 D_2$, we form the two matrices $\tilde R = R_1^{-1}R_2$ and $\tilde D \in \mathbb{R}^{N \times M-N}$ with elements 
\begin{equation}
\tilde d_{i,j} = \gamma^{j_1-j_2}\varepsilon^{2(j_2-j_1)} L^{j_2-j_1} \sqrt{\frac{j_1 !}{j_2 !}}\sqrt{2^{j_2-j_1}}.
\end{equation}
and compute their Hadamard product.  
That is why despite the harmful effects contained in $D$, the term $D_1^{-1} R_1^{-1}R_2 D_2$ does not lead to ill-conditioning.
\subsection{Truncation value $M$}
The major question arising for RBF-QR methods is the truncation value $M$. For $M=N$ we have a cheap and straightforward way of stably computing the interpolant without doing a costly QR-decomposition. Moreover, this ansatz allows for a tensor approach (cf.~\cref{sec:multivariate_tensor}) where forming full matricies for high dimensions can be avoided which is of great computational advantage.

Using $M \leq N$, the relation~\eqref{eq:CDPsi} becomes rank-deficient, since $ \mathrm{rank}(CD) < \min(M, N) = M$. Such a low-rank approximation was tested by Fasshauer and McCourt \cite[$\mathsection$ 6.1]{fasshauer2012stable} and showed rather good results. However, it still requires the assembly of a global matrix, which could be rather expensive in higher dimensions. Adding more expansion functions to reach $M=N$ significantly simplifies the structure of the method and does not harm the quality of the solution. That is why we will not be focusing on the rank-deficient case.

Since the eigenvalues of $D$ decay very rapidely, the terms $M\leq N$ become negligible for $N$ large enough, i.e.~the error is dominated by the error coming from the underlying RBF interpolation. This has also been confirmed numerically for various examples. In \cref{tab:errorDependenceJadd}, we provide the results obtained with $\varepsilon=0.1$ for one of the test functions from \cite[$mathsection$ 5.1]{rashidinia2016stable},
\begin{equation}
 f_2(x)  =  \sin\left(\frac{x}{2}\right) - 2 \cos(x) + 4\sin(\pi x), \quad x \in [-4, 4].
\end{equation}
\begin{center}
\captionof{table}{The $L_2$ interpolation error on the Chebyshev grid for the function $f_2$ with $N$ basis functions, $M = N + j_{\mathrm{add}}$ expansion functions, and 100 equidistant evaluation points.}
    \label{tab:errorDependenceJadd}
    \begin{tabular}{| c | c | c | c | c |}
    \hline
    \backslashbox{$j_{\mathrm{add}}$}{$N_{\mathrm{bf}}$} & 10 & 20 & 25 & 30 \\ \hline
    0 & 8.6629010 & 0.0029523 & 0.1937075 $\times 10^{-4}$ & 0.1827378 $\times 10^{-8}$ \\ \hline
    1 & 8.6629010 & 0.0029523 & 0.1944307 $\times 10^{-4}$ & 0.1827378 $\times 10^{-8}$ \\ \hline
    2 & 8.6648555 & 0.0029609 & 0.1944307 $\times 10^{-4}$ & 0.1836897 $\times 10^{-8}$ \\ \hline
    3 & 8.6648555 & 0.0029609 & 0.1944291 $\times 10^{-4}$ & 0.1836897 $\times 10^{-8}$ \\ \hline
    4 & 8.6648569 & 0.0029609 & 0.1944291 $\times 10^{-4}$ & 0.1836864 $\times 10^{-8}$ \\ \hline
    5 & 8.6648569 & 0.0029609 & 0.1944291 $\times 10^{-4}$ & 0.1836864 $\times 10^{-8}$ \\ \hline
    30 & 8.6648569 & 0.0029609 & 0.1944291 $\times 10^{-4}$ & 0.1836865 $\times 10^{-8}$ \\ \hline
    \end{tabular}
\end{center}

\section{Multivariate interpolation}\label{sec:multivariate}

In this section, we address the question of how to apply our stabilization algorithm to multivariate interpolation problems. First of all, we notice that the Gaussian basis is separable, i.e.~the multivariate Gaussian basis $\ensuremath{\boldsymbol\phi}_k(\x)$ (with $\x \in \mathbb{R}^d$) can be written as a product of one dimensional Gaussians,
\begin{equation}
\ensuremath{\boldsymbol\phi}_k(\x)= \exp\left(- \varepsilon^2 \|\x -\x_k^{\text{cen}}\|^2\right) = \prod_{i=1}^d \phi_k(x_i).
\end{equation}
One possibility is to derive an RBF-QR algorithm that truncates the multivariate expansion on a hyperbolic cross. 
If we use a tensor product grid of centering and collocation points, on the other hand, a very simple generalization of the stabilization algorithm can be designed by applying the HermiteGF expansion separately in each dimension. This ansatz yields a memory-sparse algorithm since it relies on Kronecker products of one dimensional matrices as we will derive in Section~\ref{sec:multivariate_tensor}. Therefore, it is particularly suitable for high-dimensional problems, even though it comes with the drawback that we loose the uniformity in all directions. Hagedorn generating functions \cite{dietert2017invariant,hagedorn2015generating} provide a truly multidimensional generalization of the HermiteGF expansion that additionally allows for anisotropic RBFs. This will be discussed in Section~\ref{sec:anisotropic}.

\subsection{Tensor product approach}\label{sec:multivariate_tensor}
For dimension $d$, let $X_{\ell}^{\mathrm{cen}}$, $\ell=1,\ldots,d$, be the centering points along each coordinate direction. Then, we can index the $d$ variate basis by a multi-index $\kb = (k_1,\ldots, k_d)$  and  write the multivariate interpolant $s(\x)$ as
\begin{equation}\begin{aligned}
s(\x) &=& \sum_{k_1=1}^{N_1} \ldots \sum_{k_d=1}^{N_d} \alpha_{\kb} \phi_{\kb}(\x) = \sum_{k_1=1}^{N_1} \ldots \sum_{k_d=1}^{N_d} \alpha_{\kb} \prod_{\ell=1}^d \phi_{k_{\ell}}(x_{\ell}) \\
&=& \left(\Phi(x_d,X_d^{\mathrm{cen}}) \otimes \ldots \otimes \Phi(x_1,X_1^{\mathrm{cen}}) \right) \ve{(\alpha)},
\end{aligned}\end{equation}
where we denote by $\ve{(\alpha)}$ the vectorization of the coefficient tensor $\alpha$. Now, we can replace $\Phi(x_{\ell},X_{\ell}^{\mathrm{cen}})$  by $H^{\gamma,\varepsilon}(x_{\ell}) B(\varepsilon, \gamma,X_{\ell}^{\mathrm{cen}})$ transforming the individual one-dimensional Gaussian bases to the HermiteGF basis with $M_{\ell}=N_{\ell}$ expansion coefficients. This yields the following expression for the interpolant,
\begin{equation}
s_M^{\gamma}(\x) = \left(H^{\gamma,\varepsilon}(x_d) B(\varepsilon, \gamma,X_d^{\mathrm{cen}})\otimes \ldots \otimes H^{\gamma,\varepsilon}(x_1) B(\varepsilon, \gamma,X_1^{\mathrm{cen}}) \right) \ve(\alpha).
\end{equation}
Introducing a second tensor product grid for the collocation points $X_{\ell}^{\mathrm{col}}$, $\ell=1,\ldots,d$, we analogously get a Kronecker product representation of the collocation matrix yielding the following expression for the expansion coefficients $\alpha$,
\begin{equation}\begin{aligned}
\ve(\alpha) =&& \left( H^{\gamma,\varepsilon}(X_d^{\mathrm{col}}) B(\varepsilon, \gamma,X_d^{\mathrm{cen}})\otimes \ldots \otimes H^{\gamma,\varepsilon}(X_1^{\mathrm{col}}) B(\varepsilon, \gamma,X_1^{\mathrm{cen}}) \right)^{-1}\\
&& \ve(f(X_1^{\mathrm{col}}, \ldots, X_d^{\mathrm{col}})).
\end{aligned}\end{equation}
Putting everything together, we get
\begin{equation}\begin{aligned}
&s_M^{\gamma}(\x) = \left(H^{\gamma,\varepsilon}(x_d) B(\varepsilon, \gamma,X_d^{\mathrm{cen}})\otimes \ldots \otimes H^{\gamma,\varepsilon}(x_1) B(\varepsilon, \gamma,X_1^{\mathrm{cen}}) \right)\\
&\left( H^{\gamma,\varepsilon}(X_d^{\mathrm{col}}) B(\varepsilon, \gamma,X_d^{\mathrm{cen}})\otimes \ldots \otimes H^{\gamma,\varepsilon}(X_1^{\mathrm{col}}) B(\varepsilon, \gamma,X_1^{\mathrm{cen}}) \right)^{-1} \ve(f(X_1^{\mathrm{col}}, \ldots, X_d^{\mathrm{col}})\\
&= \left(H^{\gamma,\varepsilon}(x_d) H^{\gamma,\varepsilon}(X_d^{\mathrm{col}})^{-1}\otimes \ldots \otimes H^{\gamma,\varepsilon}(x_1) H^{\gamma,\varepsilon}(X_1^{\mathrm{col}})^{-1}  \right)\ve(f(X_1^{\mathrm{col}}, \ldots, X_d^{\mathrm{col}})).
\end{aligned}\end{equation}
Hence, we can compute the matrices $H^{\gamma,\varepsilon}(x_{\ell})H^{\gamma,\varepsilon}(X_{\ell}^{\mathrm{col}})^{-1}$ separately for each dimension $\ell=1,\ldots,d$, and then apply them mode-wise to the tensor $f(X_1^{\mathrm{col}}, \ldots, X_d^{\mathrm{col}})$ of function values. The memory requirements for the interpolation matrices is hence limited to $dN^2$ which is much smaller than the memory requirement for the full $d$ dimensional interpolation matrix of $N^{2d}$.

    \subsection{Anisotropic approximation}
    \label{sec:anisotropic}
Until now we only considered interpolations with the same shape parameter $\epsilon$ in both directions. Given the HermiteGF-tensor structure one could also easily use different values of $\epsilon$ in different directions. Finding a stable interpolant for anisotropic multidimensional RBFs of type $\exp(-(x-x_k)^TE(x-x_k))$ is a more challenging task. A similar question was raised in \cite[$\mathsection$ 8.5]{fasshauer2012stable}, however, without further investigation. It turns out that generating function theory provides a convenient toolbox for deriving a stable basis that spans the same space, but doesn't lead to ill-conditioning related to small elements in $E$. Adapting the result of \cite[Lemma 5]{dietert2017invariant} we derive the HagedornGF expansion that is very similar to the HermiteGF expansion. 
\begin{lemma}{HagedornGF expansion}

For all positive definite $E \in \mathbb{R}^{d \times d}$, $\x_k \in \mathbb{R}^d$ the following relation holds,
 \begin{equation}
  \exp(-(\x-\x_k)^TE(\x-\x_k)) = \exp(-\x_k^T E \x_k + \x_k^T E^TE\x_k)\sum_{\lbb \in \mathbb{N}^d} \frac{(E\x_k)^{\lbb}}{\lbb!}h_{\lbb}(\x) \exp(-\x^T E \x),  
 \end{equation}
where $x_k$ is the center of the anisotopic Gaussian, $E$ is a shape matrix and $h_{\lbb}(\x)$ are tensor product of physicists' Hermite polynomials,
\begin{equation}
 h_{\lbb}(\x) = h_{\lb_1}(x_1) \cdot \ldots \cdot h_{\lb_d}(x_d).
\end{equation}

\end{lemma}
\begin{proof}
 The general Hagedorn polynomial's generating function is given by \cite[Lemma 5]{dietert2017invariant}, \cite[Theorem 3.1]{hagedorn2015generating},
 \begin{equation}
  \sum_{\lbb \in \mathbb{N}^d} \frac{t^{\lbb}}{\lbb!} q_{\lbb}(\x) = \exp(2\x^T\tb - \tb^TM\tb),
 \end{equation}
where $q^{M}_{\lbb}(\x)$ are generalized Hagedorn polynomials \cite[$\mathsection$ 3]{dietert2017invariant} that are given for any symmetric unitary matrix $M\in \mathbb{C}^{d \times d}$ by the following three-term recurrence,
\begin{equation}
(q_{\lbb+e_j}^M(\x))_{j=1}^d = 2\x q_{\lbb}^M(\x) - 2M \cdot (\lb_jq_{\lbb-e_j}^M(\x))_{j=1}^d,
\end{equation}
with boundary conditions $q^M_0 = 1$, $q_{\lb}^M = 0$ for all $\lb \notin \mathbb{N}^d$.

Consider $M = \mathrm{Id}$, $t = E\x_k$, then
\begin{equation}
 \sum_{\lbb\in\mathbb{N}^d}\frac{(E\x_k)^{\lbb}}{\lbb!}q_{\lbb}(x) = \exp(2\x^TE\x_k - \x_k^TE^TE\x_k).
\end{equation}
Note that for the case of $M = \mathrm{Id}$, Hagedorn polynomials turn into a tensor product of Hermite polynomials,
\begin{equation}
 q^{\mathrm{Id}}_{\lbb}(\x) = h_{\lb_1}(x_1) \cdot \ldots \cdot h_{\lb_d}(x_d) = h_{\lbb}(\x).
\end{equation}
Hence, we get,
\begin{align}
 &\phantom{=} \exp(-(\x-\x_k)^TE(\x-\x_k)) = \exp(-\x^TE\x + 2\x^TE\x_k - \x_k^TE\x_k) \\
 &= \exp(-\x_k^T E \x_k) \cdot \exp(\x_k^T E^T E \x_k) \cdot \exp(2\x^TE\x_k - \x_k^TE^TE\x_k) \cdot \exp(-\x^T E \x) \\
 & = \exp(-\x_k^T E \x_k + \x_k^T E^TE\x_k) \sum_{\lbb\in \mathbb{N}^d} \frac{(E\x_k)^{\lbb}}{\lbb!}h_\lb(\x)\exp(-\x^T E \x).
\end{align}

\end{proof}
An RBF-QR method can then be naturally derived based on the HagedornGF expansion. This expansion provides a new powerful tool of dealing with anisotropic approximation. However, the computational costs of that method are way higher than for the HermiteGF-tensor approach. 

Note that HermiteGF-tensor interpolation considered before corresponds to the following matrix $E$,
\begin{equation}
 E_{\mathrm{tensor}} = \begin{pmatrix}
                                 \epsilon^2 & 0\\
                                 0 & \epsilon^2
                                \end{pmatrix}.
\end{equation}
\section{Numerical results}
\label{sec:numericalResults}
In this section, we first discuss the implementation of the new method. Then, we compare the HermiteGF-based algorithm with the existing stabilization methods. We also look closer into the role of the parameter $\gamma$ in conditioning and show some multidimensional results. For all 1D tests we look at the $L_2$ error of the interpolant evaluated at 100 uniformly distributed points. For the multidimensional case less evaluation points have been used and will be specified separately below.
\subsection{Stable implementation}\label{sec:implementation}
We have implemented the HermiteGF interpolation both in \texttt{MATLAB} and \texttt{Julia}. The code can be downloaded from \url{https://gitlab.mpcdf.mpg.de/clapp/hermiteGF}. The \texttt{MATLAB} implementation has shown more stable results in some cases, on the other hand, \texttt{Julia} yields better performance (cf.~\cref{sec:performance}), especially in high dimensions where \texttt{Julia} enables easy and efficient parallelization. 

Even though the described approach allows to reduce the ill-conditioning of the collocation and evaluation matrices, the HermiteGF-based matrices still become increasingly ill-conditioned for growing number of basis functions. On the other hand, the product of the evaluation matrix $\He(X^{\mathrm{eval}})$ and the inverse of the collocation matrix $\He(X^{\mathrm{col}})$ is still well-conditioned. For this reason, it is crucial to take special care when building these matrices and inverting the collocation matrix.
The following configurations have proven to be preferable:
\begin{itemize}	
 \item For all the dimensions $\ell=1,\ldots,d$ compute $\He(X_{\ell}^{\mathrm{eval}})\He(X_{\ell}^{\mathrm{col}})^{-1}$ first, which allows to cancel out the ill-conditioning.\\
 Using the built-in operator \texttt{/} for the inversion yields good results both in \texttt{MATLAB} and \texttt{Julia}. However, \texttt{MATLAB} proved superior in the severely ill-conditioned case.
 \item The HermiteGF basis functions can be stabely evaluated by formulating them in terms of the Hermite functions $\psi_n$,
\begin{equation}
\He_n(x) =  \pi^{1/4}\psi_n(\gamma x)\exp(-\epsilon^2x^2 + (\gamma x)^2/2).
\end{equation}
Hermite functions can be stabely evaluated based on their three-term recurrence. \\
\end{itemize}
All experiments were performed with \texttt{MATLAB} if not stated otherwise.

\subsection{Comparison with existing RBF-QR methods}
In this section, we compare the performance of the above described method with the Chebyshev-QR method\footnote{Code downloaded from \url{http://www.it.uu.se/research/scientific_computing/software/rbf_qr} on November 28, 2016.} and the Gauss-QR method\footnote{Code downloaded from \url{http://math.iit.edu/~mccomic/gaussqr/} on May 29, 2017.}. 
We use two test functions that were studied in \cite[$\mathsection$ 5.1]{rashidinia2016stable}, namely
\begin{align}
f_1(x) & =  e^x \sin(2\pi x) + \frac{1}{x^2 + 1}, \quad x \in [-1, 1],\\ 
f_2(x) & =  \sin\left(\frac{x}{2}\right) - 2 \cos(x) + 4\sin(\pi x), \quad x \in [-4, 4].
\end{align}
\begin{figure}[!t]
\centering
\begin{subfigure}[t]{0.44\textwidth} 
\includegraphics[scale=0.43]{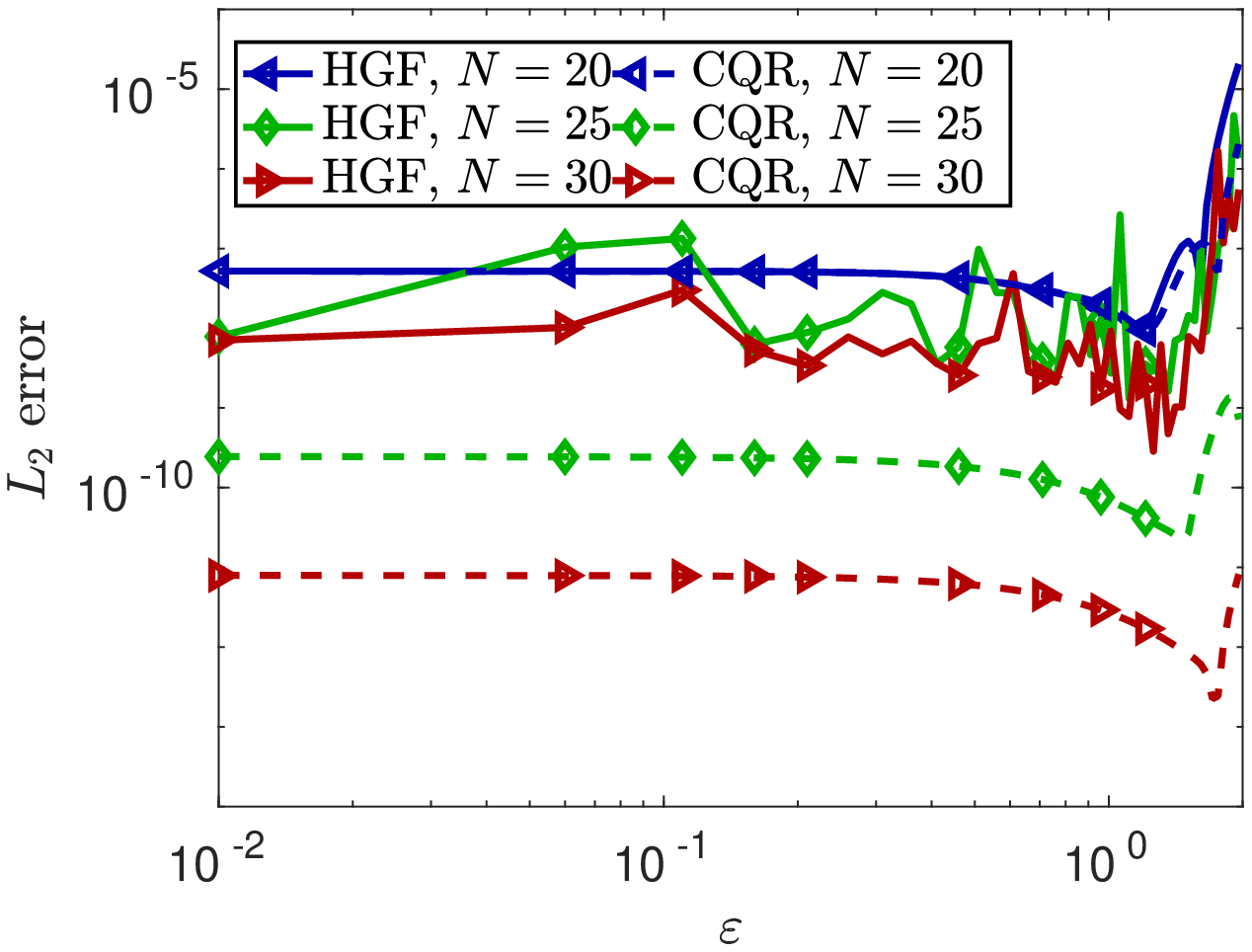}
\caption{$\gamma=1$, $\varepsilon$ = \texttt{0.01:0.05:1.99}.}
\label{fig:f1_methodsComparissonUnstable}
\end{subfigure} 
\begin{subfigure}[t]{0.44\textwidth} 
 \includegraphics[scale = 0.43]{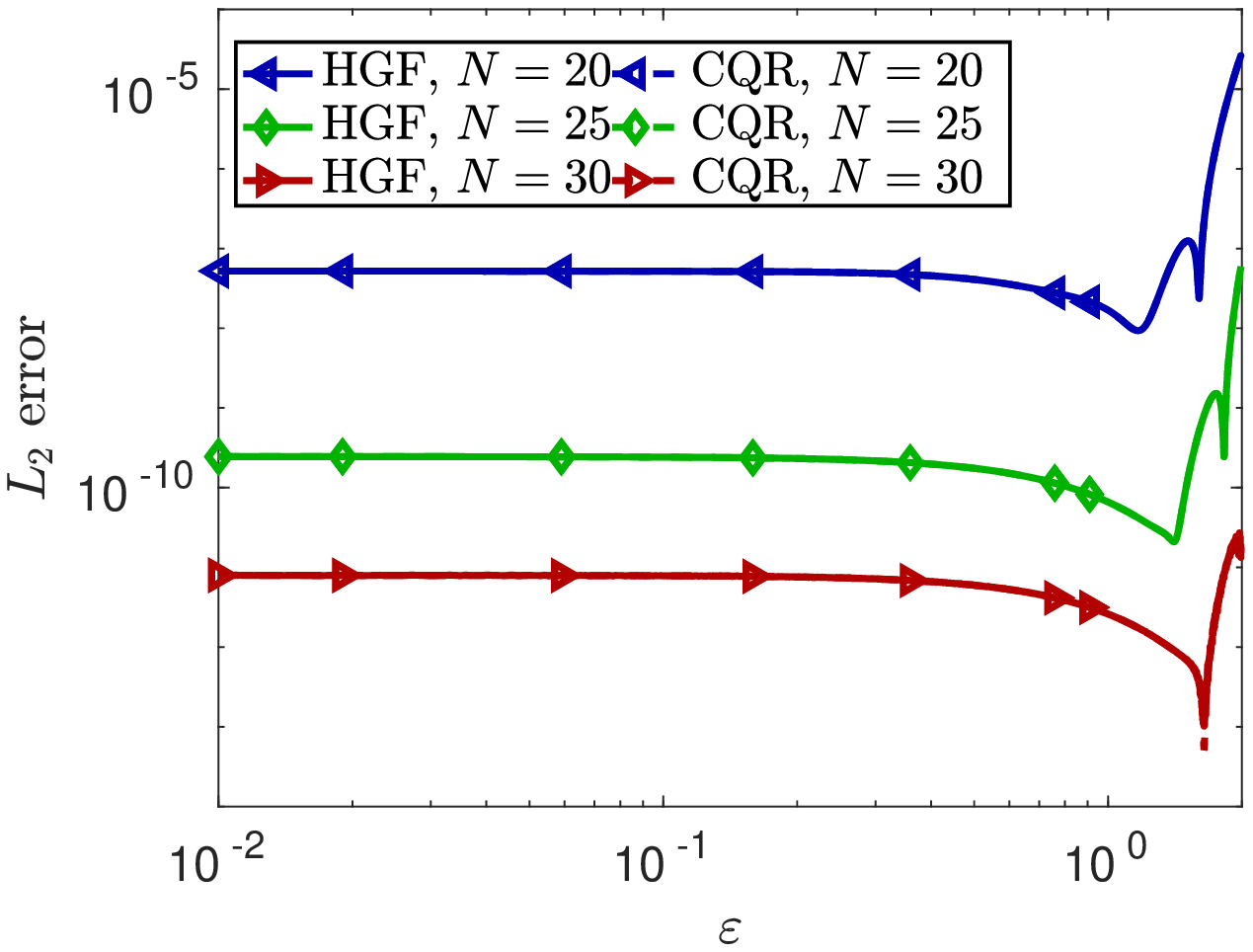}
 \caption{$\gamma=2$, $\varepsilon$ = \texttt{0.01:0.001:1.99}.}
 \label{fig:f1_methodsComparissonStable}
\end{subfigure}
\caption{For the function $f_1$ HermiteGF-tensor algorithm (HGF) tends to be unstable for $\gamma = 1$ the for larger number of Chebyshev nodes unlike the Chebyshev-QR (CQR). However the error magnitude is still reasonable. Increasing the value of $\gamma$ to 2 stabilizes the method and brings the interpolation quality in agreement with other methods.} 
\label{fig:f1_methodsComparisson}
\end{figure}
\begin{figure}[b!]
\centering
\begin{subfigure}[t]{0.44\textwidth} 
\includegraphics[scale=0.43]{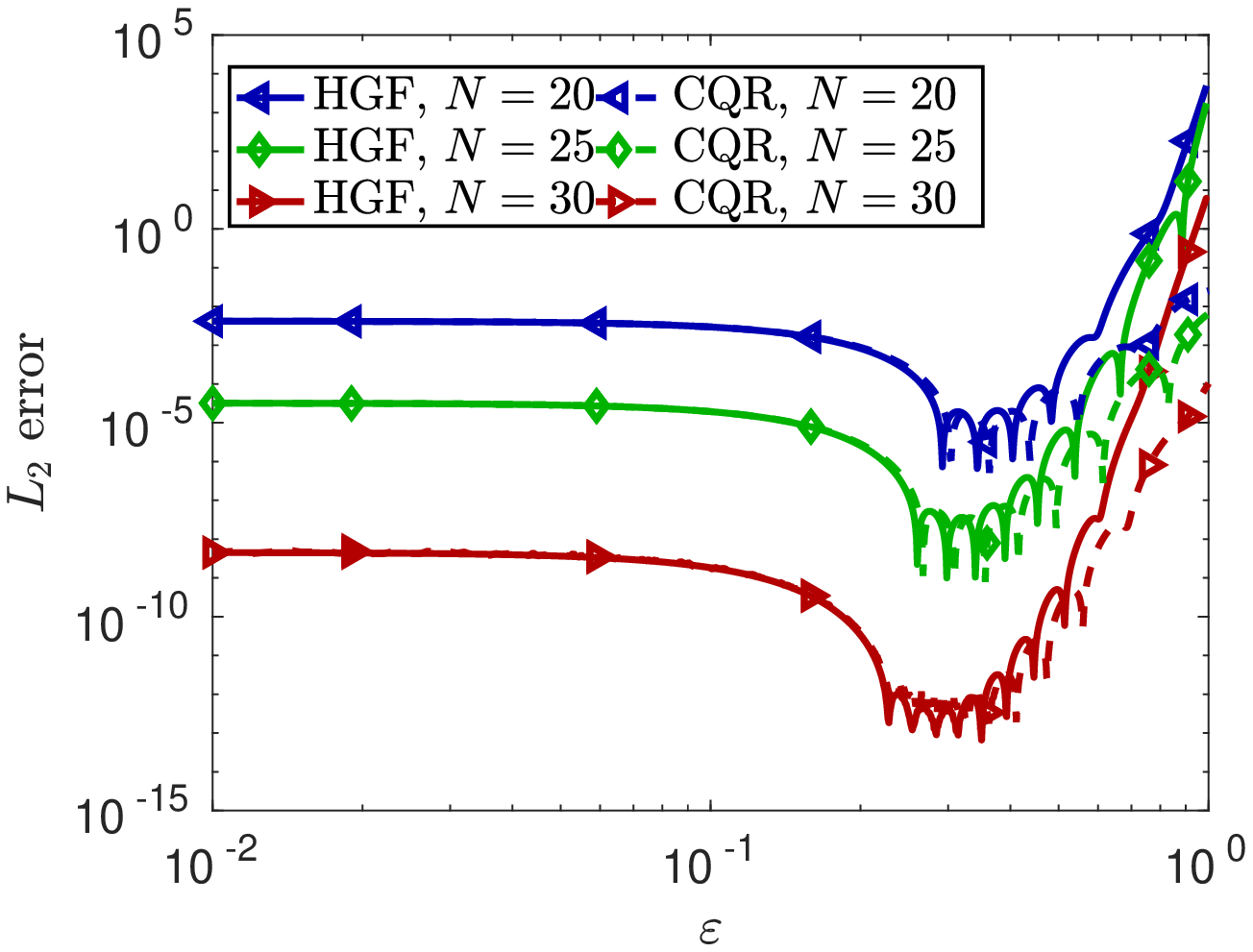} 
 \caption{$\gamma=1$, $\varepsilon$ = \texttt{0.01:0.001:0.99}.}
  \label{fig:f2_methodsComparissonGauss}
\end{subfigure} 
\begin{subfigure}[t]{0.44\textwidth} 
\includegraphics[scale=0.43]{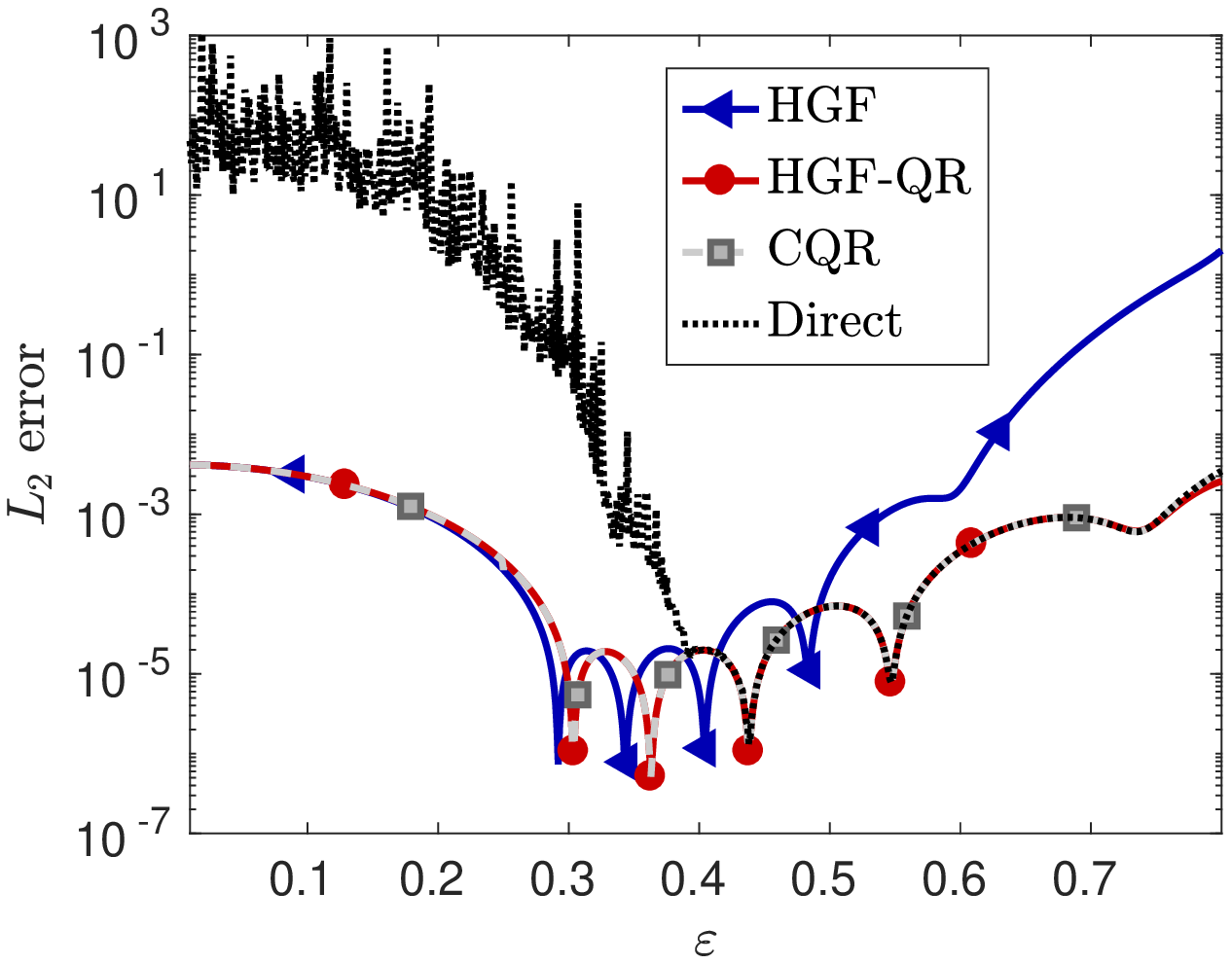}
 \caption{$\gamma=1.3$, $\varepsilon$ = \texttt{0.01:0.001:0.8}, $N = 20$. For HGF-QR $M = 30$.}
 \label{fig:f2_methodsComparissonQR}
\end{subfigure} 
\caption{Dependence of the L2 error for the function $f_2$ on the value of $\varepsilon$ with different number of Chebyshev nodes shown for the HermiteGF-tensor (HGF), the HermiteGF-QR (HGF-QR), the Chebyshev-QR (CQR), and the Gauss-QR (GQR) method. With the natural choice of $\gamma=1$ the HermiteGF-tensor method is in a good agreement with the RBF-QR methods.} 
\label{fig:f2_methodsComparisson}.
\end{figure}
Note that for the Chebyshev-QR method we must always scale the interpolated function to the unit disk. This also implies a scaling of the value of the shape parameter which we account for in \cref{fig:f2_methodsComparisson}. We use Chebyshev collocation points here but discuss the case of uniform points in \cref{sec:results_uniform}

We look at the performance of the methods for different values of $\varepsilon$. The Chebyshev-QR error curves turned out to lay exactly on top of the Gauss-QR ones with the optimal values of $\alpha$ from \cite[$\mathsection$ 5.1]{rashidinia2016stable}, that is why we only present one of them at a time.  
For all setups in the flat limit our HermiteGF-tensor method performs in a stable way unlike RBF-Direct. In \cref{fig:f1_methodsComparissonUnstable}, we see that the algorithm gets unstable for the natual choice of $\gamma=1$ for function $f_1$. However, the magnitude of the error still stays around $10^{-8}$. If we increase the value of $\gamma$ to 2, we get a full resemblance to the Chebyshev-QR results (see  \cref{fig:f1_methodsComparissonStable}).  

For the function $f_2$ with $\gamma=1$, i.e. $\gamma L=4$, HermiteGF-tensor and Gauss-QR show comparable results (see \cref{fig:f2_methodsComparisson}): For small values of $\varepsilon$ the results are identical but they start to differ slightly for the optimal $\varepsilon$ range before clearly diverging when the error starts to grow. The curve for $N=20$ where this effect is most pronounced is further investigated in \cref{fig:f2_methodsComparissonQR}.  For larger $\varepsilon$ the RBF-Direct method produces stable results that are in agreement with the Chebyshev-QR method. In the figure, we also show the results of the HermiteGF-QR method with an expansion of $M=30$ points and $\gamma=1.3$, again agreeing with Chebyshev-QR. From these experiments, we conclude that $M>N$ can be necessary in the optimal $\varepsilon$ range (especially for small $N$) to exactly reproduce the Gaussian RBF interpolant. On the other hand, the HermiteGF method with $M=N$ gives results of the same quality while being cheaper. For larger values of $\varepsilon$, the method seems to be more sensitive to the parameter choice. However, in this range the RBF-Direct algorithm would anyway be preferable.

\subsection{Scaling and conditioning}
\label{sec:scalAndCond}

Let us take a look at the behavior of the condition number of the interpolation matrix for different values of $\gamma$. We consider an interpolation matrix on an interval $[-1,1]$ as a function of the number $N_{\textrm{col}}$ of Chebyshev points. Note that the interpolation matrix that has to be inverted is independent of the interpolated function. 

\begin{figure}[h]
\centering
 \includegraphics[scale=0.45]{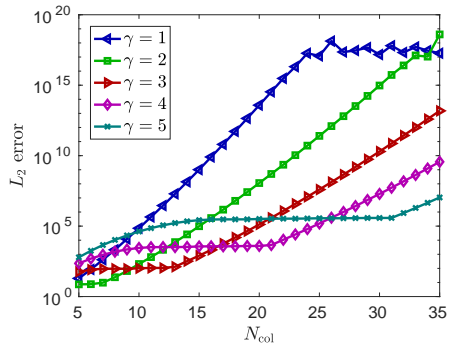}
 \caption{Conditioning of the interpolation matrix on the interval $[-1,1]$ for varying number of Chebyshev points. The condition number grows slower for larger values of the parameter $\gamma$.}
 \label{fig:conditioningVaryGamma}
\end{figure}

As we can see in \cref{fig:conditioningVaryGamma} the condition number gets smaller for larger values of $\gamma$. The larger the value of $\gamma$ the larger the evaluation interval $[-\gamma L, \gamma L]$ for the Hermite polynomials becomes. Therefore, for larger value of $\gamma$ the points are further away from each other for the same values of $N_{\textrm{col}}$, which leads to improved condition numbers. 

Recall that for the function $f_2$ from the previous section $\gamma = 1$ provided good results. However, in that case the interpolation interval was $[-4, 4]$. Therefore, the evaluation interval for Hermite polynomials is also $[-4,4]$ which corresponds to smaller condition number (equivalent to $\gamma=4$ in \cref{fig:conditioningVaryGamma}). Since for the function $f_1$, the interpolation interval was $[-1, 1]$, increasing $\gamma=2$---and hence the evaluation interval to $[-2, 2]$---reduced the condition number and allowed for stable computations for higher values of $N_{\textrm{col}}$. 

As mentioned before, the conditioning of the interpolation matrix does not depend on the interpolated function itself. However, the impact on the result can be different for different functions. Consider the following functions on the interval $[-1, 1]$ (see \cref{fig:testingFunctions}),
\begin{align}
 f^{\mathrm{c}}_1 = \cos(x^2), \quad f^{\mathrm{c}}_2 = \cos(2x^2), \quad
 f^{\mathrm{c}}_4 = \cos(4x^2).
\end{align}
We expect that the faster the function changes, especially near the boundaries, the more sensitive the interpolation quality should be towards the condition number. Indeed, looking at the $L_2$ error (see  \cref{fig:testing_gamma}) we see that for $f^{\mathrm{c}}_1$ the quality is good for all integer values of $\gamma L$ between 1 and 5.  On the other hand, for the function  $f^{\mathrm{c}}_4$ the result for $\gamma=1$ is considerably worse than for other values. Note that the values of $\gamma L$ are not fixed to integers but any $\gamma L > 0$ can be chosen. On the other hand, the stability is not sensitive to minor changes of $\gamma L$ which is why we use a rough integer estimation of the desired evaluation interval.
\begin{figure}[t]
\centering
\includegraphics[scale=0.4]{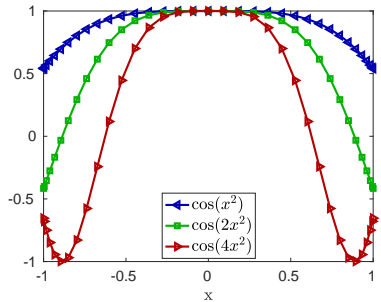}
\caption{Test functions with different gradients.}
\label{fig:testingFunctions}
\end{figure}
\begin{figure}[h]
 \begin{subfigure}[t]{0.33\textwidth} 
\includegraphics[scale=0.33]{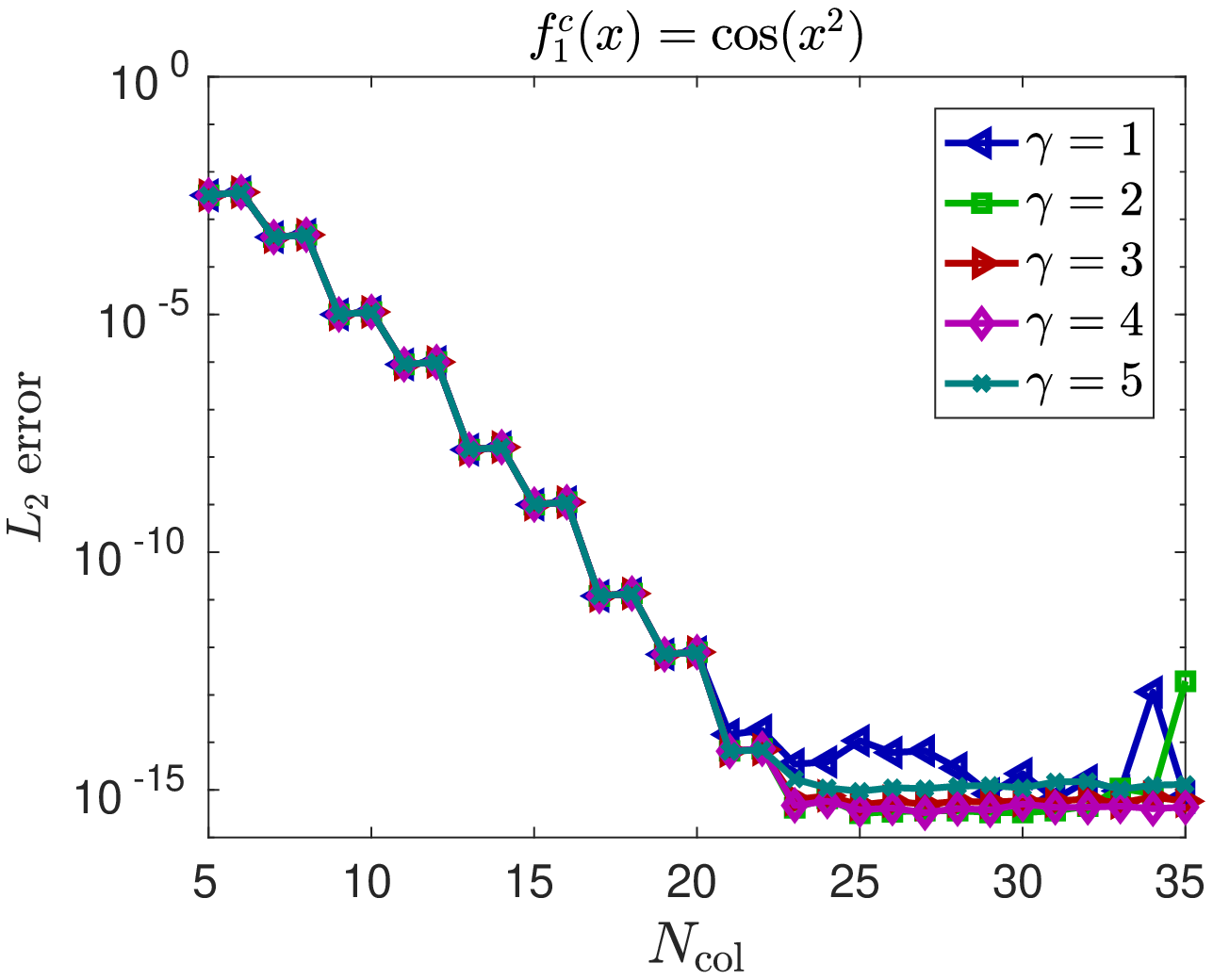} 
\end{subfigure}%
\begin{subfigure}[t]{0.33\textwidth} 
\includegraphics[scale=0.33]{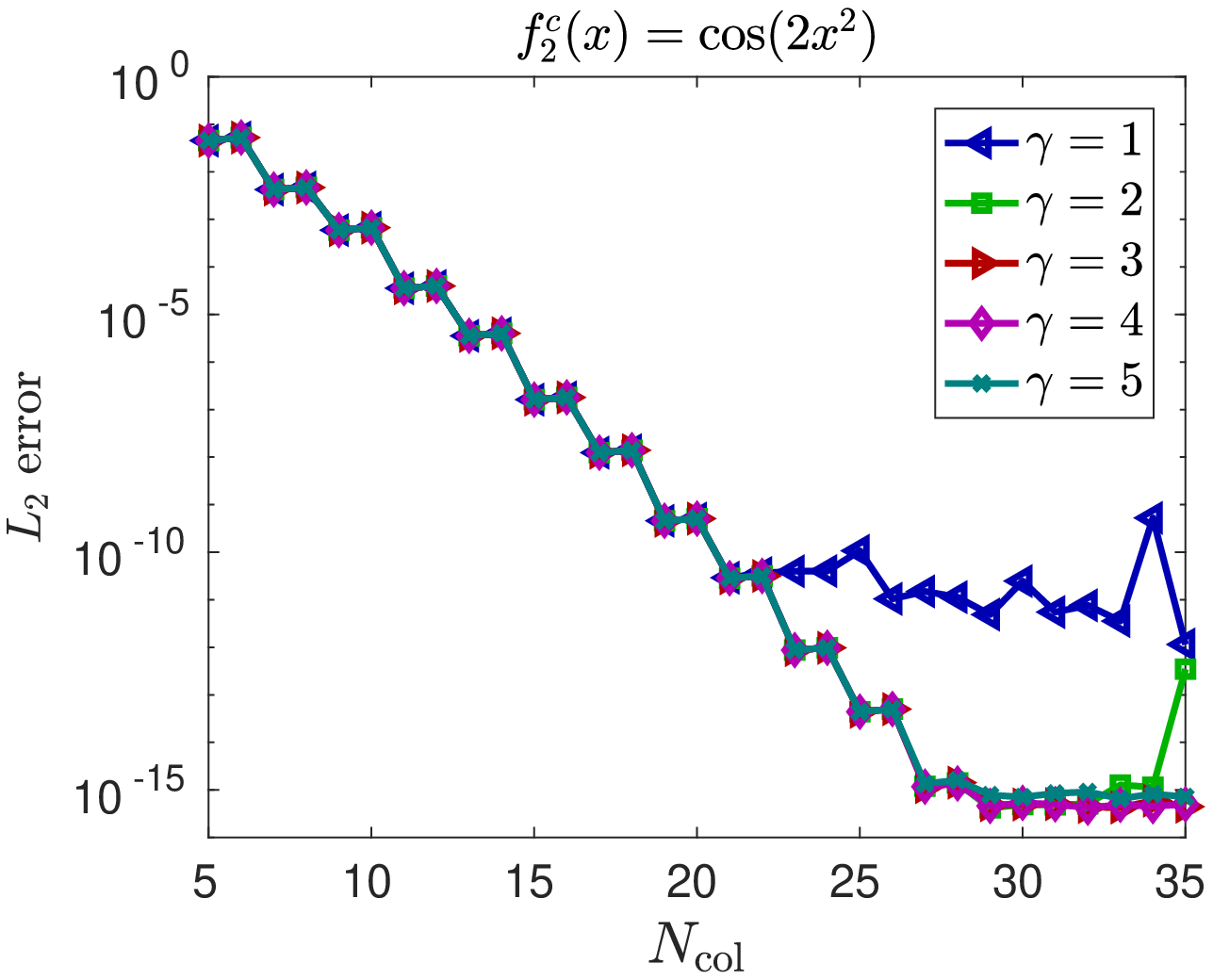}
\end{subfigure} 
\begin{subfigure}[t]{0.33\textwidth} 
\includegraphics[scale=0.33]{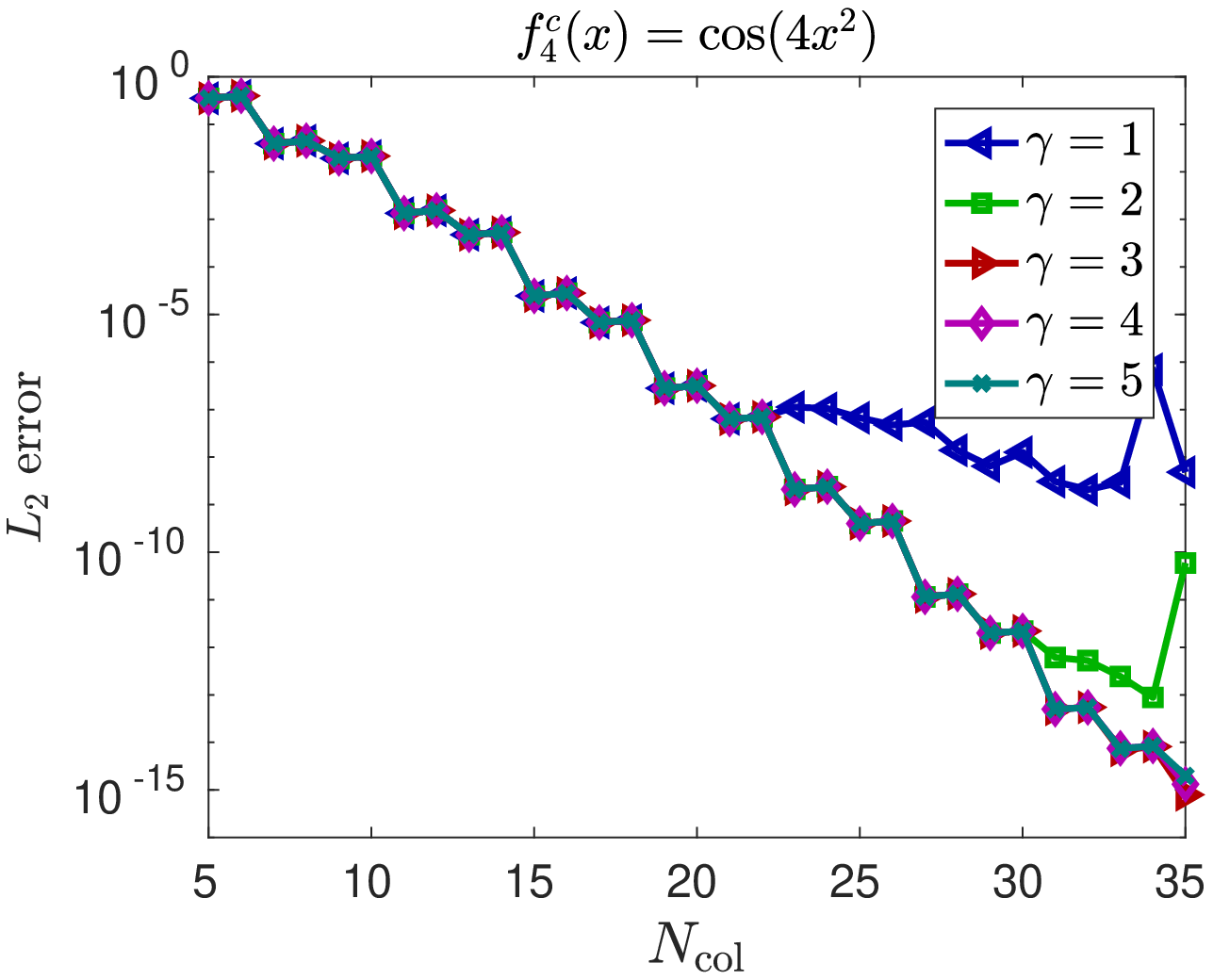} 
\end{subfigure} 
\caption{The $L_2$ error for the testing functions for different values of $\gamma$. For the more flat function $f^{\mathrm{c}}_1$ all values of $\gamma$ suit equally good. For the other two functions it is preferable to use $\gamma  \geq 3$.} 
\label{fig:testing_gamma}
\end{figure}
Even though using large $\gamma$ appears to be advantageous one should not forget that Hermite polynomials take very large values on big domains. That can lead to cancellations and overflow. From our experience, the range $[3, 5]$ seems to be optimal for $\gamma L$ for most of the cases.
\subsection{Interpolation on a uniform grid}\label{sec:results_uniform}

Spectral interpolation on uniform grids is known to be intrinsically ill-conditioned causing large errors close to the boundary \cite{platte2011impossibility}. This ill-conditioning persists after a basis transformation so that the number of basis functions $N_{\mathrm{col}}$ needs to be chosen small enough in applications where uniform nodes are of interest. We consider the following function for our tests,
\begin{equation}
 f_{\mathrm{u}}(x) = \sin(2x) + \cos(4x) + \frac{1}{2+x}, \quad x \in [-1, 1].
\end{equation}

\cref{fig:uniformGridExample_HermiteGF} shows the $L_2$ error in the interpolation of function $f_{\mathrm{u}}$ for $\varepsilon = 0.1$ as a function of the number of collocation points. The Chebyshev-QR algorithm and the HermiteGF algorithm for various values of $\gamma$ are considered. First, we note that we again need to choose $\gamma$ large enough to get results of the same quality as with the Chebyshev-QR algorithm. The results of the Chebyshev-QR algorithm also clearly show the increase of the error that is typical for uniform points (starting at $N_{\textrm{col}} = 24$). For the HermiteGF method, the error starts to decrease again as soon as numerical ill-conditioning of the interpolation matrix appears (cf.~\cref{fig:uniformGridExample_HermiteGF_cond}).

\begin{figure}[h]
\centering
\begin{subfigure}[t]{0.44\textwidth} 
\includegraphics[scale=0.42]{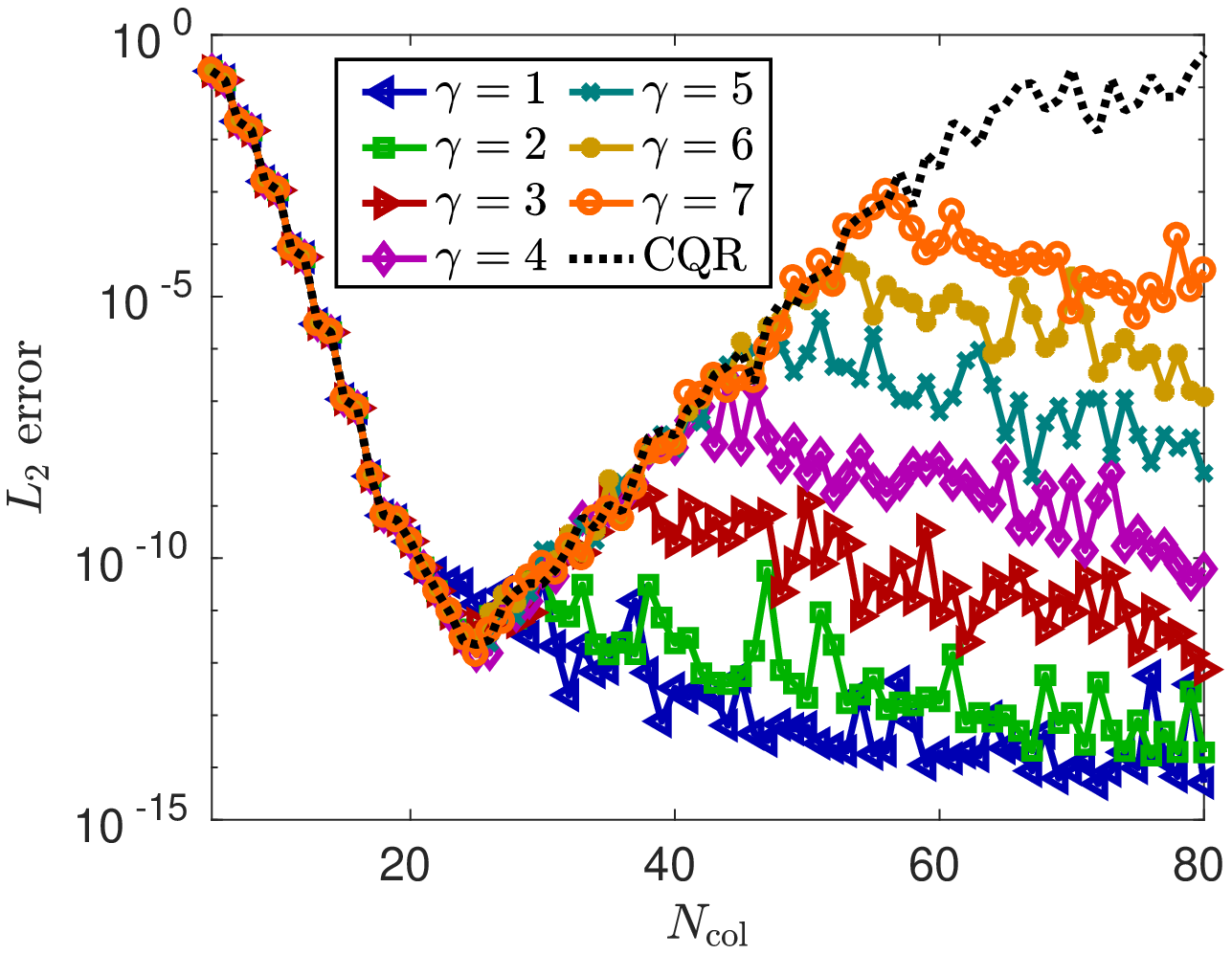} 
\caption{$L_2$ error of HermiteGF ($\gamma$ given in the legend) and Chebyshev-QR methods (CQR).}\label{fig:uniformGridExample_HermiteGF}
\end{subfigure}
\begin{subfigure}[t]{0.44\textwidth}
 \includegraphics[scale=0.42]{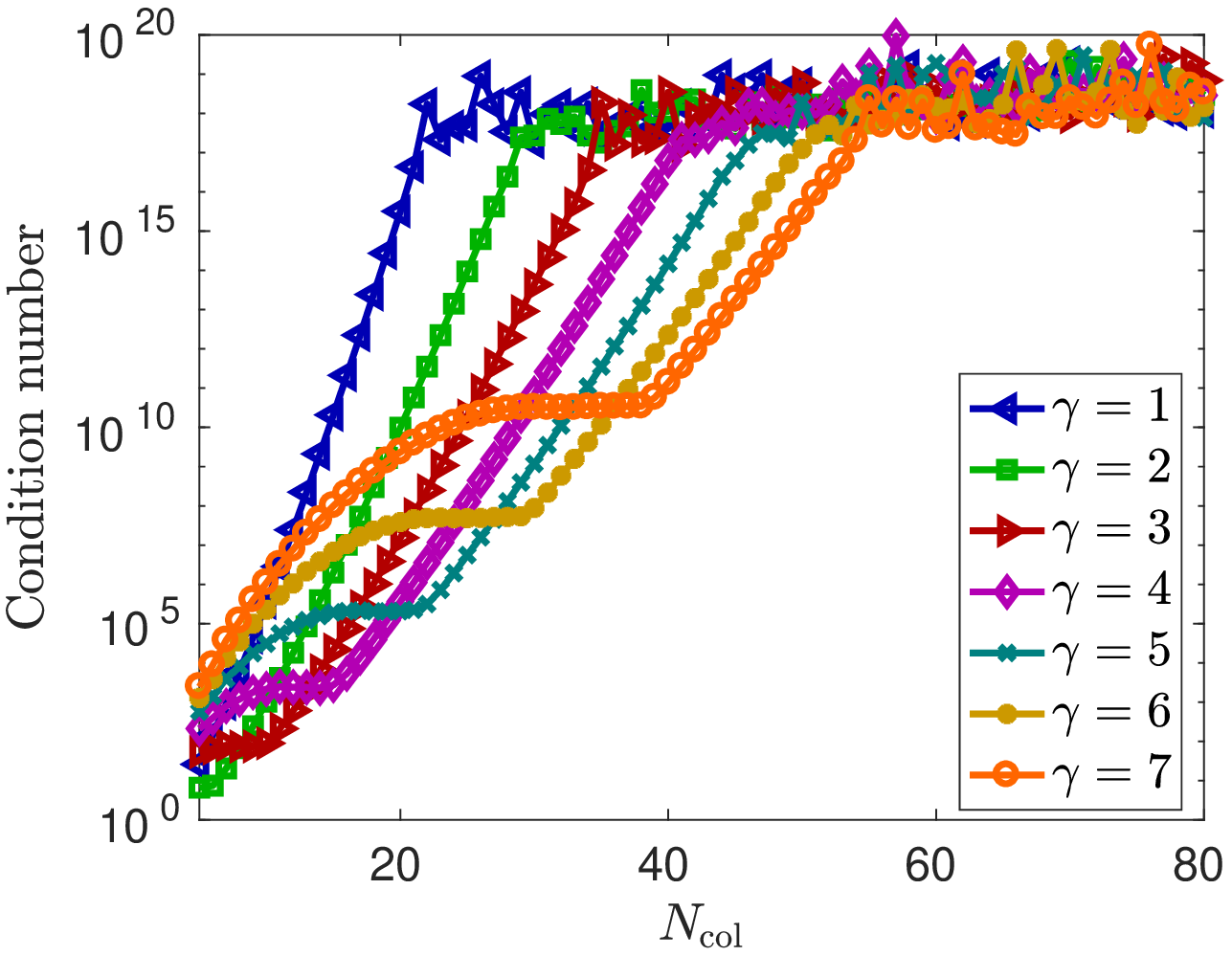}
\caption{HermiteGF, conditioning.}\label{fig:uniformGridExample_HermiteGF_cond}
\end{subfigure}

\caption{Error of in interpolation of the function $f_{\mathrm{u}}$ on a uniform grid.}
\label{fig:uniformGridExample}
\end{figure}

\subsection{Multivariate interpolation}

In this section we take a look at high dimensional interpolation. Consider the function,
\begin{equation}
 f_3(x) = \cos(\Vert x \Vert^2) 
\end{equation}
\begin{figure}[b]
 \centering
\includegraphics[scale=0.42]{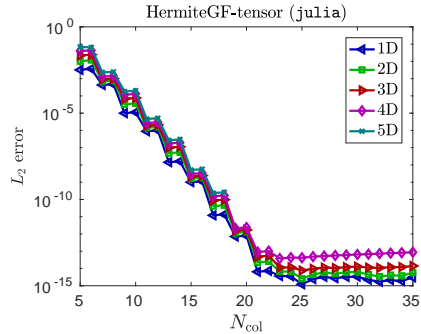}
 \caption{Dependence of the $L_2$ error of the HermiteGF-tensor ($\gamma=3$) interpolation on the number $N_{\textrm{col}}$ of Chebyshev nodes per dimension. The value of $\varepsilon$ is set to 0.1. The interpolation quality of the HermiteGF-tensor algorithm is almost dimension independent. The error has been computed on a uniform grid with 53 points per dimension.}
 \label{fig:chebyshevQR_HGF_tensor_scalability}
\end{figure}
We now look at the behavior of HermiteGF-tensor for different dimensions. With the use of simple parallelization via built-in \texttt{Julia} tools, it was possible to run tests for 1--5D. The largest simulation run contained $1.5\cdot 10^6$ points. Due to the computational complexity for 5D only 5--18 points per dimension have been considered. One can see in \cref{fig:chebyshevQR_HGF_tensor_scalability} that even though the error slightly increases with the dimension, the rate of decay of the error is the same for all dimensions. 

Note that the underlying RBF expansion used in Chebyshev-QR could be used in a similar fashion to construct a tensor based algorithm. However, the restriction to the unit domains still holds.

In order to demonstrate the potential for the HagedornGF expansion for anisotropic basis functions, we consider  the following function,
\begin{equation}
f_{\mathrm{a}}(x, y) = \cos\left(\frac{(x+y)^2}{2.88} + \frac{(y-x)^2}{4.5}\right), \quad x, y \in [-1, 1]
\end{equation}
This is an anisotropic modification of the two dimensional function $f_3$ used for the tests earlier. We expect that anisotropic interpolation should suit better in this case than a regular HermiteGF-tensor. 
For testing purposes only matrices $E$ of the following form were considered,
\begin{equation}
 E = \begin{pmatrix}
      \epsilon^2 & \xi^2\\
      \xi^2 & \epsilon^2
     \end{pmatrix}, \quad \xi < \epsilon < 1.
\end{equation}
Indeed, as one can see in \cref{fig:anisotropic_example} there exists a matrix $E$ for which the error is smaller than for the HermiteGF interpolants with equal values of $\epsilon$ in both directions.
\begin{figure}[t]
\centering
\includegraphics[scale=0.42]{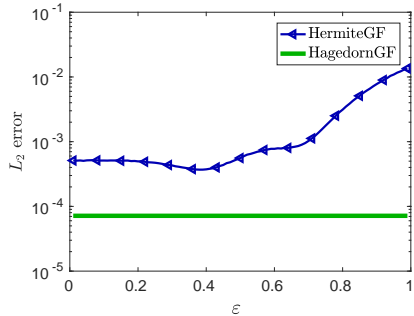}

\caption{Anisotropic interpolation of the function $f_{\mathrm{a}}$ for the positive-definite matrix $E = \left(\protect \begin{smallmatrix}
\epsilon^2 & \xi^2 \\
\xi^2 & \epsilon^2
\protect\end{smallmatrix}\right)$ with $\epsilon = 0.4$, $\xi=0.2$ on 121 Halton nodes in 2D. The result is better than all HermiteGF interpolants with the same values of the shape parameter in both directions. The error has been evaluated at 289($=17^2$) uniformly distributed evaluation points.}
\label{fig:anisotropic_example}
\end{figure}

\section{Performance tests}\label{sec:performance}
To assess the computational complexity and the performance of our HermiteGF-tensor code, we report here the run times of the \texttt{Julia} code.  Compared to \texttt{MATLAB}, \texttt{Julia} is faster for 3--5D, while the performance difference is negligible in 1--2D. The experiments where performed on the DRACO cluster of the Max Planck society. A DRACO node is equipped with \texttt{Intel 'Haswell' Xeon E5-2698v3} processors with 32 cores @2.3 GHz and 128 GB of memory. The parameters of the basis functions are set to $\varepsilon=0.1$ and $\gamma=3$. The number of evaluation points per dimension was fixed to $N_{\mathrm{eval}}=53$ in all tests.
In a first test, we split the timings to the following three essential parts of the algorithm:
\begin{itemize}
 \item Forming the interpolation matrix $\tilde{H}(X^{\col})$ and the evaluation matrix $\tilde{H}(X^{\mathrm{eval}})$;
 \item Inverting the interpolation matrix;
 \item Evaluating the interpolant $s$ at the evaluation points.
\end{itemize}
The timings for the first two tasks are shown in \cref{fig:timing_small_parts} as a function of the problem dimension for $N_{\textrm{col}}=20$ collocation points per dimension. Due to the tensor formulation of the algorithm  these first two parts do not impose significant costs. Indeed, we only need to evaluate and invert small one dimensional matrices. The costs grow linearly in the dimensionality, since we have two evaluations and one inversion of one dimensional matrices per dimension. The evaluation of the interpolant $s$, on the contrary, gets exponentially more expensive with increase of the dimensionality. That is due to the fact that we need to evaluate our interpolant in every point of the multidimensional tensor grid and the domain size grows exponentially with the dimension if we keep the amount of points per dimension constant. This can be seen from the run times reported in \cref{fig:timing_1-5D} for the total simulation times which show an exponential increase in the problem dimension. Note that the total CPU time reported in \cref{fig:timing_1-5D} steems from serial simulations in 1--3D and from parallel runs on 32 nodes for 4 and 5D. In order to minimize the influence of disturbances, we have run all serial simulations 100 times and report the minimum time. For the parallel runs, the disturbances are negligible.

As for the wall clock time, in 1--3D dimensions with moderately low amount of points (up to 35 per dimension) the interpolation can be run in less than a minute without parallelization (see \cref{fig:timing_1-5D}). For 4--5D, parallelization is required. The largest simulation ($18^5$ points in 5D) takes slightly less than a day on a full node of the DRACO cluster. 

\begin{figure}[t]
 \centering
 \begin{subfigure}[t]{0.44\textwidth} 
\includegraphics[scale=0.42]{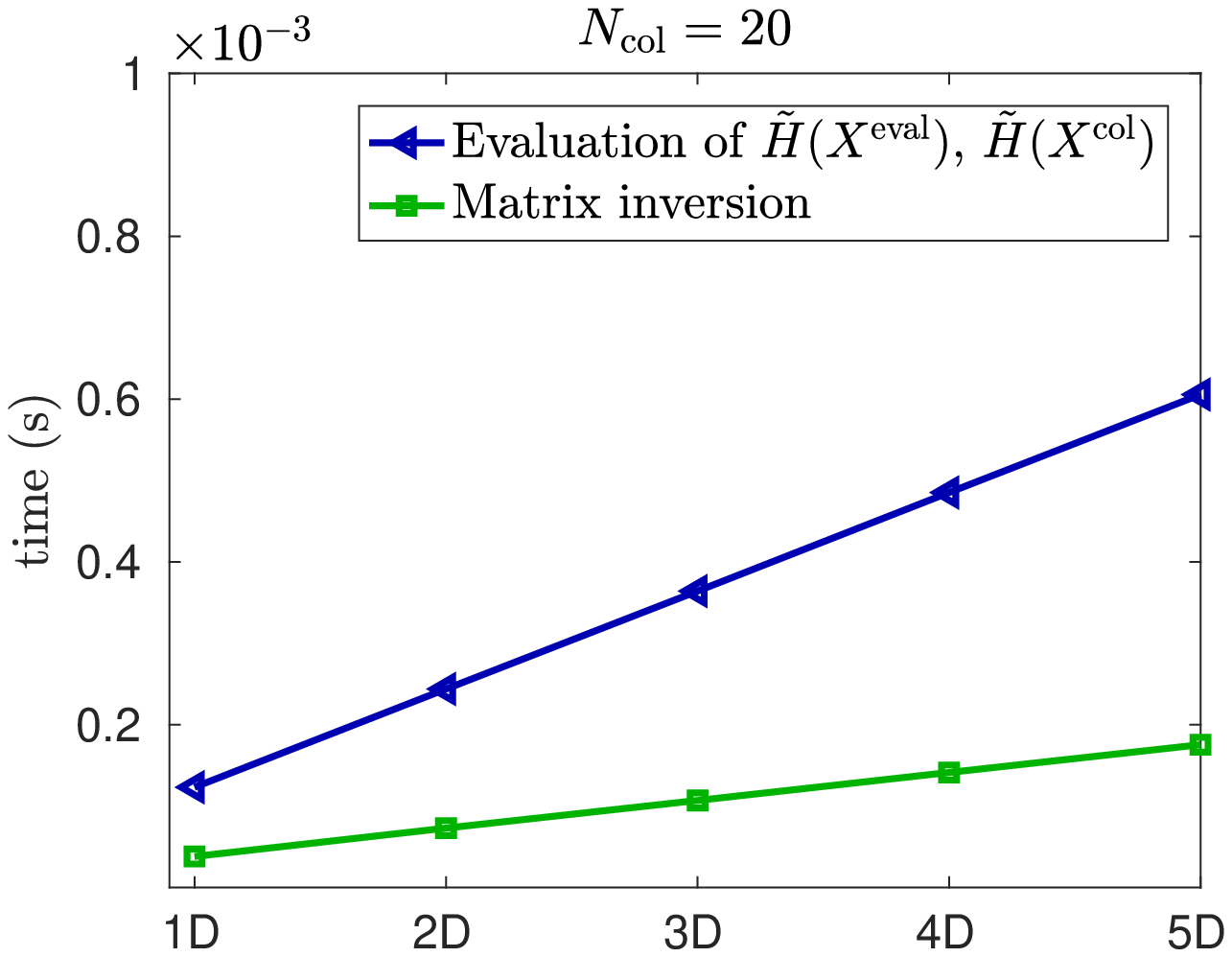}
\caption{Time of the matrices evaluation and inversion with 20 Chebyshev collocation points per dimension.}
\label{fig:timing_small_parts}
\end{subfigure} 
\begin{subfigure}[t]{0.44\textwidth} 
\includegraphics[scale=0.43]{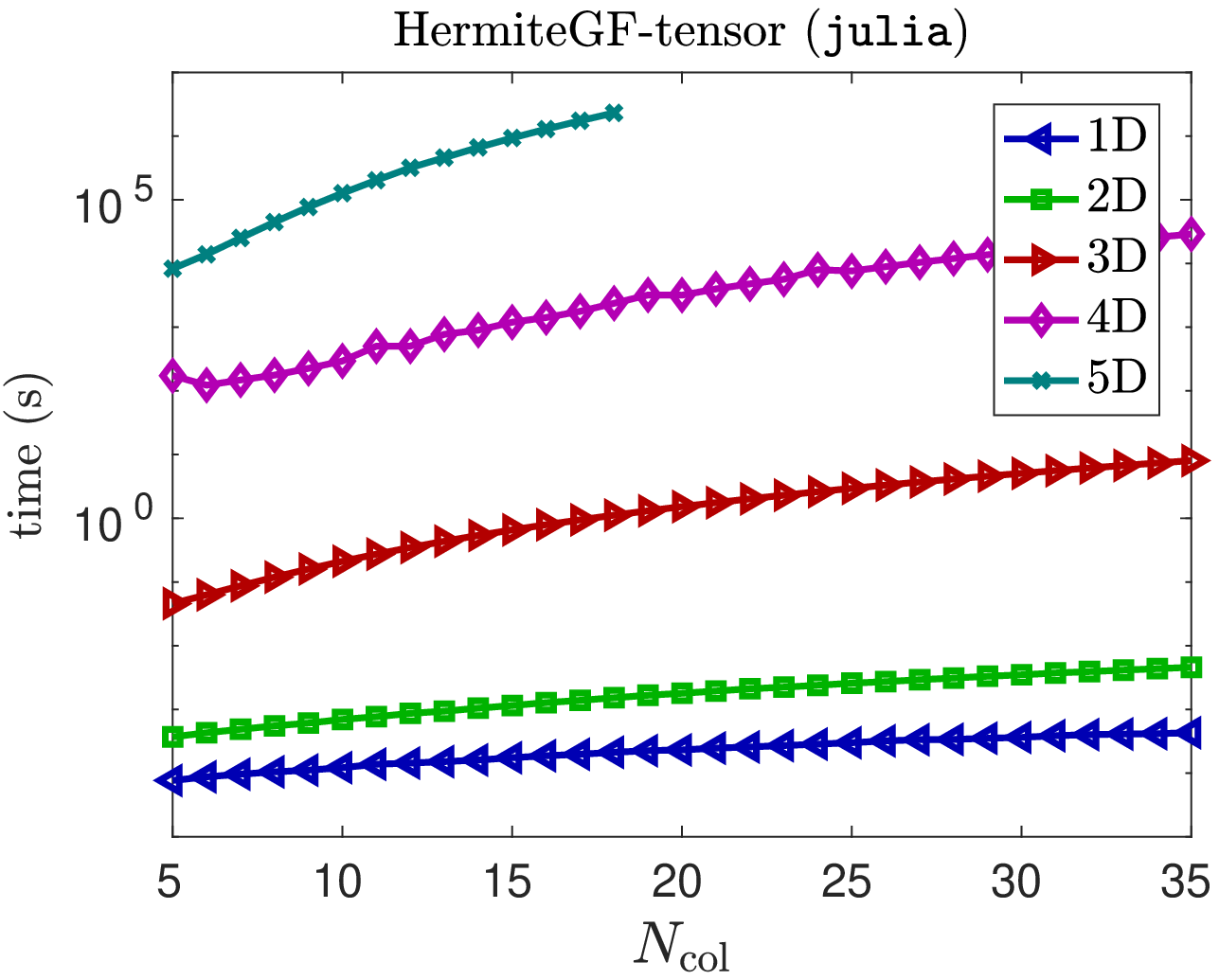}
\caption{Total CPU time.}
\label{fig:timing_1-5D}
\end{subfigure}
 \caption{The timings of the HermiteGF-tensor interpolation for 1--5D.}
 \label{fig:timings}
\end{figure}

\section{Conclusion}\label{sec:conclusions}
In this paper we derived a new stabilization algorithm for the RBF interpolation in the flat limit ($\eps \rightarrow 0$). The main idea of ``isolating'' the ill-conditioning in a special matrix is the same as in \cite{fornberg2007stable,fasshauer2012stable,fornberg2011stable}. On the other hand, we use a novel expansion of RBFs through Hermite polynomials based on the generating functions theory. Even though a standard RBF-QR approach is possible, we follow the road of choosing the number $M$ of expansion functions to be equal to $N$. This simplifies the algorithm greatly and enables an efficient implementation for up to millions of points in 5D. Compared to the existing RBF-QR stabilization methods (Chebyshev-QR and Gauss-QR) the 1D HermiteGF-based method features the same accuracy while having a simpler structure. The structure of the HermiteGF method is very similar to Gauss-QR,  however, the structure of the parameters $\eps, \gamma$ of basis functions is simpler: $\eps$ is the original shape parameter of the RBF basis and $\gamma$ stands for the size of the evaluation domain of the Hermite polynomials. The interpolation quality is not sensitive to the precise value of $\gamma$. 

Two ways to generalize the algorithm to the multivariate case were discussed. When tensor grids can be used, the HermiteGF-tensor method provides a very efficient embarassingly parallel solution. A similar solution could be also possible with the underlying RBF expansions of Chebyshev-QR and Gauss-QR algorithms. A combination with compression techniques as e.g.~proposed by Zhao \cite{Zhao16} will be explored in future work. As for the RBF-QR technique, we make a step forward by providing an opportunity for anisotropic approximations. 
The next steps in that direction is to develop an algorithm of choosing an optimal shape matrix $E$ and to use fast multipole methods to speed up the computation \cite{Yu16}.
 
The HermiteGF-tensor algorithm has been implemented both in \texttt{MATLAB} and \texttt{Julia}. The \texttt{MATLAB} code showed to be less sensitive to floating point arithmectics with large numbers. The \texttt{Julia} implementation, on the other hand, features more efficient computation. Moreover, the \texttt{Julia} built-in parallelization toolbox enabled an implementation of 5D interpolation with up to 18 points per dimension. With \texttt{Julia} being open source, it is possible to run it on any cluster. HermiteGF-tensor is currently the only available stable implementation of the RBF interpolation in the flat limit with millions of points. 

\section*{Acknowledgments}
The authors would like to thank Caroline Lasser (Technische Universit\"at M\"unchen) for constant support during the project. Fruitful discussions with Elisabeth Larsson (Uppsala University) are gratefully acknowledged. 

\providecommand{\bysame}{\leavevmode\hbox to3em{\hrulefill}\thinspace}
\providecommand{\MR}{\relax\ifhmode\unskip\space\fi MR }
\providecommand{\MRhref}[2]{%
  \href{http://www.ams.org/mathscinet-getitem?mr=#1}{#2}
}
\providecommand{\href}[2]{#2}

\end{document}